\documentclass[reqno]{amsart}

\usepackage{amsmath}
\usepackage{amsthm}
\usepackage{amsfonts}
\usepackage{amssymb}
\usepackage{enumitem}
\usepackage[colorlinks=true,linkcolor=blue,citecolor={blue}]{hyperref}
\usepackage{hyperref}
\usepackage[nocompress]{cite}
\usepackage{mathrsfs}
\usepackage{subcaption}
\usepackage{graphicx}

\DeclareMathOperator{\var}{Var}

\DeclareMathOperator{\He}{He}

\newcommand{\E}{\mathbb{E}}
\newcommand{\C}{\mathbb{C}}

\renewcommand\Re{\operatorname{Re}}

\newcommand{\eps}{\varepsilon}

\newcommand{\R}{\mathbb{R}}
\newcommand{\N}{\mathbb{N}}
\newcommand{\Z}{\mathbb{Z}}

\def\R{\mathbb{R}}

\newcommand{\DD}{\mathfrak{D}}

\theoremstyle{plain}
\newtheorem{theorem}{Theorem}[section]
\newtheorem{conjecture}[theorem]{Conjecture}
\newtheorem{lemma}[theorem]{Lemma}
\newtheorem{corollary}[theorem]{Corollary}

\newtheorem{proposition}[theorem]{Proposition}

\theoremstyle{definition}
\newtheorem{definition}[theorem]{Definition}

\newtheorem{assumption}[theorem]{Assumption}

\theoremstyle{remark}
\newtheorem{remark}[theorem]{Remark}

\newcommand{\Cr}{\color{black}}
\newcommand{\Cb}{\color{black}}
\newcommand{\nc}{\normalcolor}

\newcommand{\cedittwo}{\color{black}}

\numberwithin{equation}{section}

\begin{document}
	
	\title[Universality for roots of derivatives of entire functions]{Universality for roots of derivatives of entire functions via finite free probability}
	
		\author{Andrew Campbell}
		\address{Institute of Science and Technology Austria, Am Campus 1, 3400 Klosterneuburg, Austria}
		\email{andrew.campbell@ist.ac.at}
		\thanks{A. Campbell partially supported by ERC Advanced Grant ``RMTBeyond'' No. 101020331}
		
		\author{Sean O'Rourke}
		\address{Department of Mathematics\\ University of Colorado\\ Campus Box 395\\ Boulder, CO 80309-0395\\USA}
		\email{sean.d.orourke@colorado.edu}
		\thanks{S. O'Rourke has been partially supported by NSF CAREER grant DMS-2143142. }
		
		\author{David Renfrew}
		\address{Department of Math and Statistics\\ Binghamton University (SUNY)\\ Binghamton, NY 3902-6000\\USA}
		\email{renfrew@math.binghamton.edu}

	\begin{abstract}
	A universality conjecture of Farmer and Rhoades [\textit{Trans. Amer. Math. Soc.}, 357(9):3789--3811, 2005] and Farmer [\textit{Adv. Math.}, 411:Paper No. 108781, 14, 2022] asserts that, under some natural conditions, the roots of an entire function should become perfectly spaced in the limit of repeated differentiation.  This conjecture is known as Cosine Universality.  We establish this conjecture for a class of even entire functions with only real roots which are real on the real line.  Along the way, we establish a number of additional universality results for Jensen polynomials of entire functions, including the Hermite Universality conjecture of Farmer [\textit{Adv. Math.}, 411:Paper No. 108781, 14, 2022].  Our proofs are based on finite free probability theory.  We establish finite free probability analogs of the law of large numbers, central limit theorem, and Poisson limit theorem for sequences of deterministic polynomials under repeated differentiation, under optimal moment conditions, which are of independent interest.  
	\end{abstract}

	\maketitle 
	
	\tableofcontents

\section{Introduction} \label{sec:intro}  Tracking the effects of differential operators on polynomial roots dates back implicitly to at least Rolle's theorem and explicitly to at least Gauss's \emph{electrostatic interpretation} of critical points, which gives a simple proof and intuitive explanation of the Gauss--Lucas theorem \cite{MR1954841}. If a degree $n$ polynomial $p$ of a single complex variable has distinct roots $z_{1},\dots,z_{n} \in \mathbb{C}$, then a root $z_{*} \in \mathbb{C}$ of $p'$ must satisfy \begin{equation}\label{eq:electro interp}
	\frac{p'(z_{*})}{p(z_{*})}=\sum_{k=1}^{n}\frac{1}{z_{*}-z_{k}}=0.
\end{equation} Thus, the roots of $p'$ lie at the equilibrium points of a field created by charges at the roots of $p$. This electrostatic interpretation, while quite simple, is extremely helpful both as a proof tool and as a heuristic for tracking roots under differentiation. In fact, this heuristic helps explain an observation, first attributed to Riesz \cite{Stoyanoff1926}, that if $p$ has only real roots, then the roots of $p'$ are more evenly spaced than those of $p$. Any areas of clumping in the roots will have a repulsive effect on the critical points through \eqref{eq:electro interp}. 

In some cases this electrostatic interpretation can be extended beyond polynomials to functions with an infinite number of roots, even if the sum in \eqref{eq:electro interp} does not converge in an absolute sense. For these functions this regularizing effect of differentiation leads to a natural conjecture, due to Farmer and Rhoades \cite{Farmer-Rhoades2005}, with further refinements described by Farmer \cite{Farmer2022}, that the roots of certain entire functions should become perfectly spaced in the limits of repeated differentiation; see Section \ref{sec:univ limits of diff} for more details. 

In this work, we prove this {\cedittwo \emph{Cosine Universality Conjecture} (Conjecture \ref{conj:cos})} \Cr and the \emph{Hermite Universality \Cr Conjecture\nc} (\Cr Conjecture \nc \ref{princ:Hermite uni}) described in \cite[Principle 3.3]{Farmer2022} \nc for a large class of even entire functions through \emph{finite free probability} and the recent body of work  \cite{Steinerberger2019,Steinerberger2020,Hoskins-Kabluchko2021,Hoskins-Steinerberger2022,Arizmendi-GarzaVargas-Perales2023,Shlaykhtenko-Tao2020,Campbell-ORourke-Renfrew2024,Gorin-Marcus2020,Gorin-Kleptsyn2024}, connecting polynomial roots under differentiation to operations in random matrix theory and free probability. \Cr We additionally introduce, and prove under our assumptions, a \emph{Laguerre Universality Conjecture} (Conjecture  \ref{princ:Laguerre universality})  which implies the  conjecture of Farmer and Rhoades \cite{Farmer-Rhoades2005}.  \nc

 We believe finite free probability can play a further role in the study of entire functions in the \emph{Laguerre--P\'olya} class, i.e., functions $f$ which can be represented as \begin{equation*}
	f(z)=C_1z^{m}e^{c_1z-c_2z^2}\prod_{k=1}^{N}\left(1+\frac{z}{x_k}\right)e^{-\frac{z}{x_k}},
\end{equation*} for $C_1,c_1,x_1,x_2,\ldots\in\R$, $c_2\geq 0$, $m\in\N \cup \{0\}$, and $N\in\N\cup\{\infty\}$. Functions in the Laguerre--P\'olya class are uniform limits of polynomials, and thus enjoy many properties of polynomials, for example their roots interlace with the roots of their derivative.  As finite free probability is a newly developing field, we do not assume the reader has any familiarity with it, and provide a concise introduction in Section \ref{sec:Finite free probability}.  \Cr Our results include several finite free probabilistic limit theorems which are of independent interest, giving a generalization of both Griffin, Ono, Rolen, and Zagier \cite{Griffin-Ono-Rolen-Zagier2019} and Hoskins and Steinerberger \cite{Hoskins-Steinerberger2022}. \nc


\subsection{Universal limits of differentiation for the Laguerre--P\'olya class}\label{sec:univ limits of diff}


%
We now discuss the Cosine Universality conjecture of Farmer and Rhoades \cite{Farmer-Rhoades2005}\footnote{ A version of Cosine Universality appeared as a claimed result in \cite{Farmer-Rhoades2005}, however, as pointed out in \cite{Pemantle-Subramanian2017}, a key step contains an error. }, see also \cite{Berry2005} and \cite{Farmer2022}. 


{\cedittwo 
\begin{conjecture}[Cosine Universality, \cite{Farmer-Rhoades2005}] \label{conj:cos}
Given an entire function, $f$, which is real on the real line with only real zeros, such that the number of zeros of $f$ in $(0,r)$ and $(-r,0)$ grows sufficiently nicely, then the zeros of $f^{(n)}$ approach perfect spacing as $n\rightarrow \infty$. Furthermore, the appearance of perfect spacing arises through cosine, namely there exist real sequences $A_n,B_n,\kappa_n,D_n$ with $D_n$ bounded such that \begin{equation}\label{eq:cosine convergence}
	\lim_{n\rightarrow\infty} A_ne^{B_nz}f^{(n)}(\kappa_n z+D_n)= \cos(\pi z).
\end{equation} 
\end{conjecture}
   }

	Universal attractors of differentiation have also been explored by Berry \cite{Berry2005}, who expanded the idea to more general functions by moving away from real rooted-ness. Based on the work of Griffin, Ono,
	Rolen, and Zagier \cite{Griffin-Ono-Rolen-Zagier2019}, Farmer \cite{Farmer2022} proposed a potential refinement of Cosine Universality for even entire functions, dubbed in \cite{Farmer2022} the \emph{Hermite Universality Conjecture}, where universal attractors can be observed by the appearance of Hermite polynomials as the limits of what Farmer refers to as \emph{even Jensen polynomials}. See Section \ref{sec:Jesen polynomials and universality}, below, for details on Jensen polynomials and their universality under \Cr differential operators.\nc

	Our main tool for proving Cosine and Hermite Universality is \emph{finite free probability theory}. This theory emerged out of the celebrated works of Marcus, Spielman and Srivastava \cite{Marcus-Spielman-Srivastava2015-1,Marcus-Spielman-Srivastava2015-2,Marcus-Spielman-Srivastava2018,Marcus-Spielman-Srivastava2022,Marcus-Spielman-Srivastava2022Inter} on families of interlacing polynomials, which they used to prove the existence of bipartite Ramanujan graphs of all sizes and degrees, and to solve the Kadison--Singer problem \cite{Kadison-Singer1959}. 
	
	In addition to Cosine and Hermite Universality, we introduce \Cr a third kind of universality, \nc which we refer to as \emph{Laguerre Universality}, and demonstrate how these \Cr notions of \nc universality  have natural interpretations as (finite free) probabilistic limit theorems. Specifically, our main results in Section \ref{sec: main results uni principles} prove all three universality conjectures  for even entire functions using only growth conditions of the roots. 
	
	%

	Our probabilistic approach is motivated by the recent connection between free probability and repeated differentiation of polynomials \cite{Steinerberger2019,Steinerberger2020,Hoskins-Kabluchko2021,Hoskins-Steinerberger2022,Kabluchko2021,Arizmendi-GarzaVargas-Perales2023,Campbell-ORourke-Renfrew2024}. We prove the related finite free limits theorems for general families of real-rooted polynomials in Section \ref{sec:finite free limit theorems}, which are of independent interest, and apply these results to the universality conjectures. The statements of these theorems require no background in finite free probability, but interpreting them as probabilistic limit theorems does. 
	
	\subsection{Jensen polynomials and universality}\label{sec:Jesen polynomials and universality} Motivated partially by their role in an equivalent statement of the Riemann Hypothesis, from \cite{Polya1927}, there has been some renewed interest in the \emph{Jensen polynomials} of an entire function, $f$, which are a sequence of polynomials that uniformly approximate $f$ on compact sets.  However, as pointed out by Farmer \cite{Farmer2022} there are two commonly used Jensen polynomials. We follow the terminology from \cite{Farmer2022} to distinguish between the two possible choices. In this work, we use finite free probability to study repeated differentiation of Jensen polynomials, which, after taking limits, allows us to study repeated differentiation of entire functions.
	
	We now introduce the classical Jensen polynomials. Consider the series representation of an entire function $f$:  \begin{equation}\label{eq:series definition of f}
		f(z)=\sum_{k=0}^{\infty} \gamma_{k}\frac{z^k}{k!}.
	\end{equation}
	
	\begin{definition}\label{def:classical Jensen}
		Let $f$ be as in \eqref{eq:series definition of f}. Then, for $d,n\in\N$ the \emph{degree $d$ classical Jensen polynomial with shift $n$ of $f$} is defined by \begin{equation}
			C_{d,n}(z):=\sum_{k=0}^{d}\binom{d}{k}\gamma_{k+n}z^k.
		\end{equation}
	\end{definition} 
	
	{ See \cite{Craven-Csordas1989} and references therein for some background on Jensen polynomials.} Not only do Jensen polynomials of $f$ give a polynomial approximation of $f$, but if $f$ is an entire function of order less than $2$, then the classical Jensen polynomial with shift $n$ of $f$ approximates the $n^{th}$ derivative of $f$ in the sense that, uniformly on compact subsets we have:
	\begin{equation}
		\lim\limits_{d\rightarrow\infty}C_{d,n}\left(\frac{z}{d}\right)=f^{(n)}(z).
	\end{equation}
	Furthermore, the shifted Jensen polynomials are easy to track through differentiation. Namely, \begin{equation}\label{eq:classical diff def}
		C_{d,n}(z)=\frac{d!}{(n+d)!}\left(\frac{d}{dz}\right)^{n}C_{d+n,0}(z).
	\end{equation} In fact, one could take \eqref{eq:classical diff def} as the definition of shifted Jensen polynomials after defining the unshifted version $C_{m,0}$. 
	The Jensen polynomials are useful for studying roots of entire functions because if $f$ is an entire function of order less than $2$, then $f$ has only real zeros if and only if  $C_{d,0}$ has only real zeros for any $d\in\N$.
	
		%
		%
		For an even entire function
		\begin{equation}\label{eq:even series representation}
			f(z):=\sum_{k=0}^\infty \gamma_{2k}\frac{z^{2k}}{(2k)!}=\sum_{k=0}^\infty \eta_{k}\frac{z^{2k}}{k!},
		\end{equation} there is a more commonly used choice of Jensen polynomials, given by  what Farmer refers to as the \emph{even Jensen polynomials} of $f$. \begin{definition}\label{def:even Jensen}
			Let $f$ be as in \eqref{eq:even series representation}. Then, for $d,n\in\N$ the \emph{degree $d$ even Jensen polynomial with shift $n$ of $f$} is defined by \begin{equation} \label{eq:Jensendef}
				J_{d,n}(z):=\sum_{k=0}^{d}\binom{d}{k}\eta_{k+n}z^k.
			\end{equation}
		\end{definition} It is straightforward to see that the even Jensen polynomials of $f$ are the classical Jensen polynomials of the positive rooted function $g(z)=f(\sqrt{z})$. Hence, the even Jensen polynomials also satisfy \begin{equation}\label{eq:Jensen diff relation}
			J_{d,n}(z)=\frac{d!}{(n+d)!}\left(\frac{d}{dz}\right)^n J_{d+n,0}(z).
		\end{equation} 
		We note that in contrast to the classical case, the shifted even Jensen polynomials $J_{d,n}$ no longer have a direct connection to the $n^{th}$ derivative of $f$. To circumvent this issue we will define another set of polynomials in \eqref{eq:Even polynomials under differentiation} which will converge to $f^{(n)}$. Nevertheless, the shifted even Jensen polynomials are well studied, in particular \cite{Polya1927}, see also  \cite{Griffin-Ono-Rolen-Zagier2019}, showed that the Riemann Hypothesis is equivalent to the even Jensen polynomials of the Riemann $\Xi$-function having real zeros for all $d$.

		 In \cite{Griffin-Ono-Rolen-Zagier2019}, the team of Griffin, Ono, Rolen, and Zagier prove that the even Jensen polynomials converge to Hermite polynomials as $n\rightarrow\infty$ given certain assumptions on the coefficients $\eta_{k+n}$. This is of particular interest in that it implies that $J_{d,n}$ is hyperbolic, i.e., it has only real roots, for sufficiently large $n$. They then apply this to the Riemann $\Xi$-function.  Motivated by the result of \cite{Griffin-Ono-Rolen-Zagier2019}  for the Riemann $\Xi$-function, Farmer \cite[Principle 3.3]{Farmer2022} proposed the following. 
		\begin{conjecture}[Hermite Universality, \cite{Farmer2022}]\label{princ:Hermite uni}
			For a large class of functions $f$ (such as those considered in the Cosine Universality conjecture discussed above), there should exist sequences $\mathcal{A}_n,\mathcal{B}_n$, and $\mathcal{C}_n$ such that uniformly on compact subsets\begin{equation}
				\lim_{n\rightarrow\infty}\mathcal{A}_{n} J_{d,n}\left(\mathcal{C}_nz+\mathcal{B}_{n} \right)=\He_{d}(z),
			\end{equation} where $\He_{d}$ is the degree $d$ Hermite polynomial \begin{equation}\label{eq:Hermite def}
				\He_{d}(z):=\sum_{k=0}^{\lfloor \frac{d}{2}\rfloor}\frac{d!(-1)^{k}}{k!(d-2k)!}\frac{z^{d-2k}}{2^{k}}.
			\end{equation}
		\end{conjecture} Given that Conjecture \ref{princ:Hermite uni} and the results of \cite{Griffin-Ono-Rolen-Zagier2019} are formulated for even functions, 
		we will focus on the even Jensen polynomials and not the classical Jensen polynomials of $f$. In Theorem \ref{thm:Hermite uni}, we verify Conjecture \ref{princ:Hermite uni} and give explicit formulas for $\mathcal{A}_n,\mathcal{B}_n$, and $\mathcal{C}_n$ in terms of the coefficients $\{\eta_{k}\}_{k=0}^\infty$ from \eqref{eq:even series representation}. 
		However, we do not see a direct connection between Conjecture \ref{princ:Hermite uni} and Cosine Universality, as taking $n \to \infty$ leads to information about derivatives of $g$, not $f$. Additionally, despite the fact that in certain limits, the Hermite polynomials converge to cosine, we do not see how the Hermite polynomials are related to the even Jensen polynomials of cosine. In fact, \cite{Dimitrov-Youssef2009} showed that the Hermite polynomials are not the Jensen polynomials of any function in the Laguerre--P\'olya class. However, the Laguerre polynomials, which are related to the Hermite polynomials, see \eqref{eq:Laguerre -1/2 def} and Proposition \ref{prop:cosine as finite free poisson}, do appear naturally in our approach.
		
		
		To make the connection between  even Jensen polynomials  and Cosine Universality more explicit, we propose, and verify in Theorem \ref{thm:Laguerre uni}, the following universality conjecture for even Jensen polynomials. 
		\begin{conjecture}[Laguerre Universality]\label{princ:Laguerre universality}
			For a large class of functions $f$, the unshifted even Jensen polynomials of $f^{(2n)}$ converge, after rescaling, to the generalized Laguerre polynomials \begin{equation}\label{eq:Laguerre -1/2 def}
				L_{d}^{\left(-\frac{1}{2}\right)}(z)=\sum_{k=0}^{d}\frac{(-1)^{k}\left(d-\frac{1}{2}\right)_{d-k} }{k!(d-k)!}z^k.
			\end{equation}
		\end{conjecture} To see why Laguerre Universality is directly connected to Cosine Universality we point out the following proposition whose proof follows from direct computation of coefficients. As we will see in Section \ref{sec:Finite free probability}, the even Jensen polynomials of cosine have an extremely natural position in finite free probability theory. 
		
		\begin{proposition}\label{prop:cosine as finite free poisson}
			Let $\{J_{d,0}\}_{d=0}^{\infty}$ be the even Jensen polynomials of $\cos(z)$ of shift $0$ (so that $J_{d,0}$ is a polynomial of degree $d$). Then \begin{equation}
				J_{d,0}(z)=4^{d}\frac{(d!)^2}{2d!}L_{d}^{\left(-\frac{1}{2}\right)}\left(\frac{z}{4}\right)
			\end{equation} and \begin{equation}
				J_{d,0}(z^{2})=(-2)^{d}\frac{d!}{2d!}\He_{2d}\left(\frac{z}{\sqrt{2}}\right),
			\end{equation}
			where $\He_{2d}$ and $L_{d}^{\left(-\frac{1}{2}\right)}$ are defined in \eqref{eq:Hermite def} and \eqref{eq:Laguerre -1/2 def}, respectively. 
		\end{proposition}

		For Hermite Universality, \eqref{eq:Jensen diff relation} provides the means for computing the $n\rightarrow\infty$ limit of $J_{d,n}$. In order to study the even Jensen polynomials of $f^{(2n)}$, we define the differential operator \begin{equation}\label{eq:M definition}
			M:=2\left(D+2zD^2 \right),
		\end{equation} where $D=\frac{d}{dz}$ and $z$ is the multiplication operator. 
		By noting that for $g(z)=f(\sqrt{z})$ so that \[ f''(\sqrt{z}) = 4z g''(z) + 2 g'(z) = M g(z),\] it follows that if $W_{d,n}$ are the unshifted even Jensen polynomials of $f^{(2n)}$, then we have the relationship \begin{equation}\label{eq:Even polynomials under differentiation}
			W_{d,n}(z)=\frac{d!}{(n+d)!}M^nW_{n+d,0}(z),\quad W_{m,0}(z)=J_{m,0}(z).
		\end{equation} 
		We will always use $J_{d,n}$ to denote the degree $d$ even Jensen polynomial of $f$ with shift $n$ and $W_{d,n}$ the degree $d$ even unshifted Jensen polynomial of $f^{(2n)}$. The former corresponds to Hermite Universality and the latter to Laguerre Universality.
		
		%
			%
			
			We will prove Cosine Universality (Theorem \ref{thm:Cosine uni}) using Laguerre Universality (Theorem \ref{thm:Laguerre uni}). We now present a heuristic argument, that uses the double limits in $n$ and $d$ and for the sake of presentation ignores constants:
			\begin{equation*}
				f^{(2n)}(z)=g_{n}(z^2)\overset{d\gg1}{\approx} W_{d,n}\left(\frac{z^2}{d}\right)
				\overset{n\gg 1}{\approx}d!L^{(-\frac{1}{2})}_{d}\left(\frac{z^2}{4d}\right)
				=\He_{2d}\left(\frac{z}{\sqrt{2d}} \right)
				\overset{d\gg1}{\approx}\cos(z), 
			\end{equation*} where $g_n(z)=f^{(2n)}(\sqrt{z})$. 
			The first $\approx$ follows because $W_{d,n}$ are the classical Jensen polynomials of $g_n(z)$ and the second $\approx$ is Laguerre Universality. The following equality and last $\approx$ are well-known properties of Hermite polynomials.
			Our proof avoids taking the double limit but morally follows this argument.


			We conclude this section with an outline of the remainder of the paper. In Section \ref{sec:main results}, we formally state our Cosine, Hermite, and Laguerre Universality theorems, as well as our finite free limit theorems for derivatives of polynomials. In Section \ref{sec:Finite free probability}, we present the necessary background on finite free probability. In Section \ref{sec:proofsuni}, we prove the Cosine, Hermite, and Laguerre Universality theorems, and in Section \ref{sec:limit theorems proofs} we prove the finite free limit theorems for derivatives of polynomials. In Section \ref{sec:proof_of_coefficients} we prove a technical result concerning the coefficients of entire functions.  We conclude with examples in Section \ref{sec:examples}.
			
	{\cedittwo		\subsection{Notation} We employ asymptotic notation $O(\cdot)$, $o(\cdot)$, $\lesssim$, $\sim$, etc.\ under the assumption that some sequence index, such as $n$ or $m$, tends to infinity. We write $c_n=O(b_n)$ or $c_n \lesssim b_n$ if there exists some constant $C>0$ such that $|c_n|\leq C|b_n|$ for all $n > C$, $c_n=o(b_n)$ if $\frac{c_n}{b_n}\rightarrow0$, and $c_n\sim b_n$ if $\frac{c_n}{b_n}\rightarrow1$.

			\section{Main results}\label{sec:main results} We divide our main results into two categories. The first, given in Section \ref{sec: main results uni principles}, concerns the Cosine, Hermite, and Laguerre Universality Conjectures (Conjectures \ref{conj:cos}, \ref{princ:Hermite uni} and \ref{princ:Laguerre universality}). Confirming the latter two of these universality conjectures is an application of our second category of main results, given in Section \ref{sec:finite free limit theorems}, which we call finite free limit theorems of differentiation. However, the statements of these limit theorems require no knowledge of finite free or even classical probability. Readers familiar with finite free probability will be able to see why we describe these as a law of large numbers, central limit theorem, and Poisson limit theorem. 
			
			\subsection{Results on the universality of entire functions and Jensen polynomials}\label{sec: main results uni principles} 
			We begin with the exact conditions we wish to assume about the even entire function $f$. {\cedittwo For $r > 0$, we denote by $n_+(r)$ the number of real roots of $f$ in the interval $[0,r]$.} 
			
			\begin{assumption}\label{assump:root density}
				We say an even entire function $f$ satisfies Assumption \ref{assump:root density} if $f$ is order less than $2$, real on the real line, only has real roots, and \Cb $n_+$ is regularly varying, meaning there exists $\alpha\in (0,2)$ such that for every $x>0$ \begin{equation}\label{eq:root density assump}
						\lim\limits_{r\rightarrow\infty}\frac{n_{+}(xr)}{n_{+}(r)}=x^{\alpha}.
				\end{equation} \nc  
			\end{assumption} 
			
			{\cedittwo  In Lemma \ref{lem:coefficients from roots}, we show our assumption on the roots in \eqref{eq:root density assump}, implies the growth of the coefficients of $f$ is governed by $\alpha$.}
			\Cr The parameter $\alpha\in (0,2)$ determines the order of the function $f$ and the scale on which Cosine, Laguerre, and Hermite Universality hold.  The condition \eqref{eq:root density assump} implies that $n_+$ is \emph{regularly varying}, a property that arises naturally in probability theory, and we refer the reader to \cite{Resnick87} and references therein for more information. \nc \Cb For $f$ satisfying Assumption \ref{assump:root density}, we define the generalized inverse $n_{+}^{-1}(y)=\inf\{r>0: n_{+}(r)\geq y \}$ and sequence $a_{n}>0$ by 
			\begin{equation} 
				\label{eq:defan}
				a_{n}=\left[\frac{1}{2n}n_{+}^{-1}\left(\frac{2n}{\pi\alpha\csc(\pi\alpha/2)}\right)\right]^2.
			\end{equation}  Since  $n_{+}$ is regularly varying, by Assumption \ref{assump:root density} $a_{n}$ is also regularly varying, see \cite[Proposition 0.8]{Resnick87}, and thus we can write $a_{n}=\widetilde{h}(n)n^{-2+\frac{2}{\alpha}}$ for some \emph{slowly varying} function $\widetilde{h}$, i.e.\ $\widetilde{h}$ satisfies \eqref{eq:root density assump} with $\alpha=0$.\nc 
			
			
			
			We now state our main universality results, whose proofs are given in Section \ref{sec:proofsuni}.  While we state it for even functions, the first result can be applied to odd functions by first taking one derivative. 
		
			\begin{theorem}[Cosine Universality for even functions]\label{thm:Cosine uni}
				Let $f$ be an even entire function that satisfies Assumption \ref{assump:root density}. \Cr Then, Conjecture \ref{conj:cos} holds, specifically, \nc \begin{equation}\label{eq:converge to cosine}
					\lim\limits_{n\rightarrow\infty}\frac{1}{  f^{(2n)}(0)  } f^{(2n)}(\sqrt{a_n} z)=\cos(z),
				\end{equation} uniformly on compact subsets of the complex plane, where $a_n$ is given in \eqref{eq:defan}.
			\end{theorem}

			\begin{remark}
			\Cr	Convergence to cosine under repeated differentiation  has previously been proven for the Riemann $\Xi$-function and some generalizations by Ki \cite{Ki2006}. Gunns and Hughes \cite{Gunns-Hughes2019} proved this conjecture for functions in the extended Selberg class; {\cedittwo both papers rely on the special form of the Fourier transform of such functions.} Pemantle and Subramanian \cite{Pemantle-Subramanian2017}  proved this conjecture for a random function whose roots are distributed according to a homogeneous Poisson point process, {\cedittwo using the particularly nice factorization of such functions to recover useful coefficient information. Our approach
is similar, in that we use the nice factorization available for even functions, but our finite free limit theorems require such mild conditions that we are able to consider general collections of deterministic roots. 

Additionally, these previous results, as well as} the original conjecture of \cite{Farmer-Rhoades2005}, were all stated for functions  where $\alpha=1$. Our result proves that perfect spacing occurs (on scales determined by $\alpha$) for functions up to order $2$. 
			
			Beyond order $2$, repeated differentiation will create roots away from the real line \cite{MR2296054} and there is no reason to expect perfect spacing of roots on $\R$.  		
			We believe some regularity on $n_{+}$ is required, and the regularly varying condition appears natural given \eqref{eq:electro interp}. Specifically, a critical point $z_c$ of $f$ should satisfy \begin{equation*}
				\frac{f'(z_c)}{f(z_c)}=0,
			\end{equation*} and this fact should drive the dynamics of the roots towards perfect spacing. Convergence of $\frac{1}{n}\frac{f'(nz)}{f(nz)}$ on $\C\setminus\R$ is essential to the proof in \cite{Pemantle-Subramanian2017}. We use only one of the directions, but \eqref{eq:root density assump} is in fact necessary and sufficient for $\frac{1}{n}\frac{f'(b_nz)}{f(b_nz)}$ to converge for some sequence $b_n>0$. It is possible that Theorem \ref{thm:Cosine uni} holds under a weaker condition on $n_{+}$, but it is not clear to us how to determine the scale $a_n$ without the regular variation assumption.
			\end{remark}  \nc
		
			Theorem \ref{thm:Cosine uni} will be proven as a corollary of the following theorem.
			\begin{theorem}[Laguerre Universality]\label{thm:Laguerre uni}
				Let $f$ be an even entire function, as represented in \eqref{eq:even series representation}, that satisfies Assumption \ref{assump:root density} and fix $d\in\N$.  Let $W_{d,n}$ be the even Jensen polynomials of $f^{(2n)}$, as in \eqref{eq:Even polynomials under differentiation}. Then, \Cr Conjecture \ref{princ:Laguerre universality} holds, specifically, \nc 
				\begin{equation}\label{eq:even Jensen Lag limit}
					\lim\limits_{n\rightarrow\infty} \frac{1}{\gamma_{2n}}W_{d,n}\left(a_nz\right)= 4^{d}\frac{(d!)^2}{2d!}L_{d}^{\left(-\frac{1}{2}\right)}\left(\frac{z}{4}\right),
				\end{equation} uniformly on compact subsets of the complex plane, where $a_n$ is given in \eqref{eq:defan}. Moreover, the polynomials on the right-hand side of \eqref{eq:even Jensen Lag limit} are the even Jensen polynomials of $\cos(z)$.
			\end{theorem}

			\begin{remark}
				\Cr To our knowledge no version of Theorem \ref{thm:Laguerre uni} has been previously stated in the literature before our work. The motivation for Theorem \ref{thm:Laguerre uni} is to directly connect the universal limits of $f^{(n)}$ to universal limits of even Jensen polynomials.   \nc
			\end{remark}
			
			We additionally prove the Hermite Universality Conjecture for functions satisfying Assumption \ref{assump:root density}. This result follows from an application of Theorem \ref{thm:CLT}, stated below, which shows that Hermite polynomials naturally appear after repeated differentiation of high degree polynomials. 
			
			\begin{theorem}[Hermite Universality]\label{thm:Hermite uni}
				Let $f$ be an even entire function, as represented in \eqref{eq:even series representation}, that satisfies Assumption \ref{assump:root density} and let $d\in\N$. Let $J_{d,n}$ be the degree $d$ even Jensen polynomial with shift $n$ of $f$, as in \eqref{eq:Jensendef}. Define the centering and normalization sequences \begin{equation}\label{eq:Herm uni shift}
					b_n:=-2(2(n+d)-1)\frac{\gamma_{2(n+d-1)}}{\gamma_{2(n+d)}},
				\end{equation} and \begin{equation}
					c_n:=4(2(n+d)-1)^2\left[\left(\frac{\gamma_{2(n+d)-2}}{\gamma_{2(n+d)}}\right)^2-\frac{2(n+d)-3}{2(n+d)-1}\frac{\gamma_{2(n+d)-4 }}{\gamma_{2(n+d)}} \right].
				\end{equation} Then, \Cr Conjecture \ref{princ:Hermite uni} holds, specifically, \nc uniformly on compact subsets of the complex plane, \begin{equation}\label{eq:Herm uni}
					\lim_{n\rightarrow\infty} \frac{c_{n}^{-d/2}}{\gamma_{2(n+d)}}J_{d,n}\left( {\sqrt{c_n}z+b_n} \right)=\He_{d}(z).
				\end{equation}
			\end{theorem}
			
			\begin{remark}The centering and normalization sequences in Theorem \ref{thm:Hermite uni} have quite natural, though not obvious, probabilistic interpretations. If $X$ is a random variable uniformly distributed on the roots of $J_{n+d,0}$, then $b_{n}=\E X$ and $c_{n}\sim { \frac{\var(X)}{n+d} }$. Thus, they center and normalize $X$ to have mean $0$ and variance approximately {$n+d$}.  
		\end{remark}
			
			\begin{remark}
				\Cr The work of \cite{Griffin-Ono-Rolen-Zagier2019} has generated a substantial amount of interest in similar directions for Jensen polynomials, including \cite{Larson-Wagner2019,OSullivan2022Monat,OSullivan2021,Wagner2022,Griffin--Ono--Rolen--Thorner--Tripp--Wagner2022,Wagner2020,Iskander-Jain-Talvola2020}. {\cedittwo However, these previous results rely on having (either by assumption or a nice enough representation of the Fourier transform) a good expansions of 
\begin{equation} \label{eq:logp} \log\left(\frac{ p_{n+j} }{ p_n}  \right)  \end{equation}
as a polynomial in $j$ with $n$-dependent coefficients, where $p_n$ are the coefficients of the Jensen polynomials.} \Cr We instead prove that the assumption \eqref{eq:root density assump} implies an expansion for \eqref{eq:logp}, with $j=1,2$, then using finite free probability and convexity we show this suffices to prove our results, for a general class of even functions. 
			\end{remark}
			
			Farmer \cite[Section 5]{Farmer2022} points out some numerical evidence suggesting that for any fixed $d$, the rate of convergence to the $d$-th Hermite polynomial is $1/\sqrt{n}$. We shall see that Theorem \ref{thm:Hermite uni} is a special case of Theorem \ref{thm:CLT}, our central limit theorem (CLT) for repeated differentiation. To minimize the assumptions on $\gamma_{2m}$, we prove this CLT under the weakest possible assumptions on the growth rate of the roots. However, under stronger assumptions our proof can be adapted to get a Berry--Esseen type theorem which would verify Farmer's observation. The analogous rate of convergence for the finite free CLT under the finite free additive convolution appears in \cite{Arizmendi-Perales2020}. For ease of presentation, we omit the details.
			
			\subsection{Finite free limit theorems of differentiation}\label{sec:finite free limit theorems} In this section we present the limit theorems for generic sequences of real rooted monic polynomials \begin{equation}\label{eq:original sequence defintion}
				P_m(z)=\sum_{k=0}^{m}(-1)^{k}a_{k,m}z^{m-k}.
			\end{equation} 
			Theorems \ref{thm:Laguerre uni} and \ref{thm:Hermite uni} will follow by specializing these limit theorems to Jensen polynomials.
{\cedittwo Which in turn, we use to prove limit theorems for derivatives of entire functions in the Laguerre--P\'oyla class.
To the best of our knowledge, this work is the first instance of finite free probability being applied to study a function with an infinite number of roots. Given that functions in the Laguerre--P\'oyla class are well approximated by polynomials with real roots we believe finite free probability should serve as a powerful new tool in the study of such functions. Particularly, in the case of universal limits, where only a few root statistics (such as average spacing) govern the limiting behavior (such as limiting spacing of high derivatives).   }

			We define the empirical root measure of the polynomial $P_m$ to be the probability measure \begin{equation}
				\mu_{P_m}:=\frac{1}{m}\sum_{z:P_m(z)=0} \delta_{z},
			\end{equation} 
			where the roots are counted with multiplicities.  
			We denote the moments and absolute moments of $\mu_{P_m}$ by 
			\begin{equation}\label{eq:moment def}
				m_j(P_m):=\int_\R z^j\ d\mu_{P_m}(z)\quad\text{  and  }\quad|m|_{j}(P_m):=\int_\R |z|^j\ d\mu_{P_m}(z)
			\end{equation} 
			for $j \in \mathbb{N}$. 
			We present below a finite free law of large numbers, central limit theorem, and Poisson limit theorem in terms of the moments of $P_m$. The assumptions on moments of $\mu_{P_m}$ in these theorems can be translated into assumptions on the coefficients of $P_m$ by using Newton's identities. { In particular, our results only rely on the first few moments, which can be directly read off from the largest coefficients of the polynomial.} For positive integers $n$ and $j$, $(n)_{j}$ denotes the Pochhammer sequence $(n)_j = \prod_{i=1}^j (n-i+1)$.
			
			
			\begin{theorem}[Law of large numbers]\label{thm:Law of large numbers}
				Assume $m_{1}(P_{m})\rightarrow a\in\R$ and $m_2(P_{m})=o(m)$ as $m\rightarrow\infty$. Then, for any $d\in\N$, \begin{equation}\label{eq:LLN}
					\lim_{n\rightarrow\infty} \frac{1}{(n+d)_{n}}\left(\frac{d}{dz} \right)^{n}P_{d+n}(z)=(z-a)^d,
				\end{equation} uniformly on compact subsets of the complex plane.
			\end{theorem}

			We define the dilation function on polynomials $\mathfrak{D}_k$  for $k>0$ by \begin{equation}\label{eq:dilation def}
				\DD_kp(z):=k^{\deg(p)}p(z/k).
			\end{equation} The roots of $\DD_kp$ are the roots of $p$ multiplied by $k$, with $\DD_kp$ and $p$ having identical leading-degree coefficients. 
			
			\begin{theorem}[Central limit theorem]\label{thm:CLT}
				Fix $d\in\N$. Assume $m_1(P_m)=o(m^{-1/2})$, $m_2(P_m)\rightarrow 1$ and $|m|_{2+\eps}(P_m)=o(m^{\eps/2})$ as $m\rightarrow\infty$ for some $\eps>0$. Then, \begin{equation}
					\lim_{n\rightarrow\infty}\frac{1}{(n+d)_{n}}\left(\frac{d}{dz} \right)^{n}\DD_{\sqrt{n+d}}P_{d+n}(z)=\He_d(z),
				\end{equation} uniformly on compact subsets of the complex plane, where $\He_{d}$ is the $d$-th Hermite polynomial \eqref{eq:Hermite def}.
			\end{theorem}
			
			\begin{remark}
				 After a preprint of this article first appeared, it came to our attention that Theorem \ref{thm:CLT} with $\eps=1$ is equivalent to Theorem 2.9 in \cite{Gorin-Kleptsyn2024}, which is proven using completely different methods. 
							\end{remark}
							
				{\cedittwo
				Both Theorem \ref{thm:Law of large numbers} and Theorem \ref{thm:CLT} are optimal in terms of moments.  It is straightforward to show $P_{m}(z)=z^{m}-\frac{m^2}{4}z^{m-2}$ does not satisfy the higher order moment bounds of Theorem \ref{thm:Law of large numbers}, and the derivatives do not converge to $z^d$. A similar example which does not satisfy the hypothesis nor conclusion of Theorem \ref{thm:CLT} is $P_{m}(z)=z^{m}-\frac{m}{2}z^{m-2}$. It is worth noting that in neither Theorem \ref{thm:Law of large numbers} nor Theorem \ref{thm:CLT} do we assume that the empirical measures have any limit, merely that the first few moments do not grow too quickly in the degree. 
				
			Theorem \ref{thm:CLT} is exactly a deterministic version of the result of Hoskins and Steinerberger \cite{Hoskins-Steinerberger2022} for polynomials $p_n(z) = \prod_{i=1}^n (z-X_i)$ with the roots, $X_i$, being independent and identically distributed (iid) random variables. In fact, our result replaces their assumption that all moments of $X_i$ are finite with the assumption that $ \E[|X_i|^{2+ \eps}]$ is finite. Indeed, under this assumption, by the Law of Large Numbers, almost surely, the $(2+\eps)^{th}$ moment of the empirical root measure, $|m|_{2+\eps}$, is bounded, so Theorem \ref{thm:CLT} can be applied to recover their result.
			
				The results in both \cite{Griffin-Ono-Rolen-Zagier2019} and the papers \cite{Hoskins-Steinerberger2022, Gorin-Kleptsyn2024} concern polynomials converging to Hermite polynomials under repeated differentiation. However, these two groups of papers share no citations on MathSciNet and neither cites the other. We did not find indications the two groups or authors of subsequent works were aware the repeated differentiation results were being applied in such similar ways in two different fields. All three results are versions of Theorem \ref{thm:CLT} (though \cite{Griffin-Ono-Rolen-Zagier2019} make no assumptions on the real-rootedness of the function). The approach in \cite{Hoskins-Steinerberger2022} uses the independence of the roots and Newton's identities to apply the classical law of large numbers and central limit theorem to recover the Hermite polynomials in the limit of repeated differentiation under stronger moment assumptions, while \cite{Gorin-Kleptsyn2024} uses a contour integral approach.  Theorem \ref{thm:CLT} instead follows the intuition from \eqref{eq:electro interp} that after a large number of derivatives the {remaining polynomial only depends on few simple empirical statistics of the roots.}
				
				}

			We give the following corollary for characteristic polynomials of random matrices, which is interesting in its own right. 
			
			\begin{corollary}\label{cor:limits of Wigner characteristic functions}
				Let $\{W_n\}_{n=1}^\infty$ be a sequence of $n\times n$ Wigner matrices, i.e., $W_n$ is an $n \times n$ Hermitian matrix with mean $0$ and variance $1$ entries that are independent up to the symmetry condition. We additionally assume the entries have finite moments of all orders\footnote{One could instead assume finite fourth moment, but we choose this assumption for simplicity.}. If $d\in\N$ and $\Phi_n(z)=\det(z-W_n)$ is the characteristic polynomial of $W_n$, then almost surely  \begin{equation}
					\lim\limits_{n\rightarrow\infty}\frac{1}{(n+d)_n}\left(\frac{d}{dz} \right)^{n}\Phi_{n+d}(z)=\He_{d}(z)
				\end{equation} uniformly on compact subsets of the complex plane.
			\end{corollary}
			
			Finally, we present a Poisson limit theorem for polynomials, \Cr which we prove under optimal moment conditions, \nc which we will use to prove the Laguerre Universality Conjecture for even Jensen polynomials of functions with real roots. 
			
			\begin{theorem}[Poisson limit theorem]\label{thm:Poisson limit theorem}
				Assume the polynomials $P_{m}$ are monic, have only non-negative roots, $\frac{m_1(P_m)}{m}\sim a\in (0,\infty)$ and $m_2(P_m)=o(m^3)$ as $m \to \infty$. Then, \begin{equation}
					\lim\limits_{n\rightarrow\infty} \frac{1}{(2n+2d)_{2n}}M^n \DD_{a^{-1}} P_{n+d}(z)=d!(-1)^{d}L_{d}^{\left(-\frac{1}{2}\right)}(z),
				\end{equation} uniformly on compact subsets of the complex plane, where $M$ is defined in \eqref{eq:M definition} and $L_{d}^{\left(-\frac{1}{2}\right)}$ is defined in  \eqref{eq:Laguerre -1/2 def}.
			\end{theorem}
			
			\begin{remark}\label{rem:General laguerre}
				By a line-by-line modification of the proof, a generalization of Theorem \ref{thm:Poisson limit theorem} holds for any operator $M_{\alpha,t}=t\left((1+\alpha)D+zD^{2} \right)$ where $t>0$ and $\alpha>-1$ and the limit replaced by $L^{\left(\alpha\right)}_{d}$, the generalized Laguerre polynomial of parameter $\alpha$.  However, for notational simplicity, we will just prove Theorem \ref{thm:Poisson limit theorem} and leave the details to the reader.
				
				\Cr As with Theorem \ref{thm:Laguerre uni} we are not aware of any previous version of Theorem \ref{thm:Poisson limit theorem} appearing in the literature before, even under weaker moment assumptions. \nc
			\end{remark}
			
			The following is the Poisson analogue of Corollary \ref{cor:limits of Wigner characteristic functions} for positive definite random matrices. We omit the proof for brevity, as it follows directly from Theorem \ref{thm:Poisson limit theorem} and standard random matrix results.
			
			\begin{corollary}\label{cor:limits of Wishart char polynomials}
				Let $\{X_n\}_{n=1}^\infty$ be a sequence of $n\times n$ random matrices with iid entries that are mean $0$, variance $1$, and have finite moments of all orders. If $d\in\N$ and $\Psi_n(z)=\det(z-X_{n}^*X_{n})$ is the characteristic polynomial of $X_{n}^*X_{n}$, where $X_n^*$ is the conjugate transpose of $X_n$, then almost surely  \begin{equation}
					\lim\limits_{n\rightarrow\infty}\frac{1}{(2n+2d)_{2n}}M^n\Psi_{n+d}(z)=d!(-1)^{d}L_{d}^{\left(-\frac{1}{2}\right)}(z),
				\end{equation} uniformly on compact subsets of the complex plane, where $M$ is defined in \eqref{eq:M definition} and $L_{d}^{\left(-\frac{1}{2}\right)}$ is defined in  \eqref{eq:Laguerre -1/2 def}.
			\end{corollary}
			
			\subsection{A note on complex roots} We formulated all of our results in terms of real rooted polynomials. However, given the results on the Riemann $\Xi$-function \cite{Ki2006,Griffin-Ono-Rolen-Zagier2019} one may inquire if similar results hold for functions with a small number of complex roots or only under the assumption that the roots are restricted to some strip. This restriction to a strip is motivated by a result of Kim \cite{Kim1996} that if $f$ is an entire function of order less than $2$ which is real on the real line with all of its roots in a strip containing the real line, then for any $R>0$ the roots of $f^{(n)}$ in $|z|<R$ are all real for $n$ sufficiently large.
			
			We chose to focus instead on real rooted functions where the theory of finite free probability is much more developed. However, much of what we do does not require real roots. The two main applications of real-rootedness in our arguments are in the proof of Lemma \ref{lem:coefficients from roots}, \Cr below, \nc and in bounding the higher moments for the empirical root measures of our polynomials. If one assumed conditions on the coefficients of $f$ (similar to \cite{Griffin-Ono-Rolen-Zagier2019}) and bounds on all the absolute moments of the polynomials in Section \ref{sec:finite free limit theorems}, then one could adapt our techniques to allow for some complex roots.  We leave this direction for future research. 

			\section{Background on finite free probability theory}\label{sec:Finite free probability}
			
			
			
			The proofs of our main results are based on finite free probability theory.  In this section, we provide a basic introduction to the concepts and results in finite free probability theory that we will need.  While many of the concepts in finite free probability are based on those from free probability theory, the reader does not need any background in free probability theory to understand the proofs.  We refer the interested reader to \cite{Nica-Speicher2006, Mingo-Speicher2017} for more details on free probability.

			In \cite{Marcus-Spielman-Srivastava2022}, the related field of finite free probability was introduced by defining a convolution on { monic} polynomials that gives the expected characteristic polynomial of a random matrix. More precisely, if $A$ and $B$ are $n$-dimensional Hermitian matrices with characteristic polynomials $p$ and $q$, respectively, then the finite free additive convolution of $p$ and $q$ is given by:
			\[ p(z) \boxplus_n q(z) := \E_Q[\chi_{A+QBQ^T}(z) ], \]
			where the expectation is taken with respect to $Q$, a Haar-distributed orthonormal matrix, and $\chi_{A+QBQ^T}$  is the characteristic polynomial of $A+QBQ^T$. In fact, if 
			\begin{equation} \label{eq:polypq}
				p(z) = \sum_{i=0}^n z^{n-i} (-1)^i a_i^p \quad \text{ and } \quad q(z) = \sum_{i=0}^n z^{n-i} (-1)^i a_i^q, 
			\end{equation}
			are { monic polynomials} then $ p(z) \boxplus_n q(z) $ can be computed explicitly as
			\[ p(z) \boxplus_n q(z) = \sum_{k=0}^n z^{n-k} (-1)^k \sum_{i+j=k} \frac{ (n-i)! (n-j)!   }{n! (n- i -j)!   } a_i^p a_j^q .\]
			It was observed in  \cite{Marcus-Spielman-Srivastava2022} that $\boxplus_n$ can also computed in terms of differential operators; namely if $\widehat p$ and $\widehat q$ are such that $\widehat p(D) z^n= p(z) $ and $\widehat q(D)z^n = q(z) $, where $D$ denotes the differentiation operator, then 
			\begin{equation} \label{eq:sumderiv}
				p(z) \boxplus_n q(z) =  \widehat p(D)  \widehat q(D)  z^n. 
			\end{equation}
			Remarkably, this convolution was originally introduced by Walsh \cite{MR1501220} in 1922 in a different context, and it enjoys many nice properties; see the discussions in \cite{MR225972,MR1954841} for further details and historical notes.

			In \cite{Marcus-Spielman-Srivastava2022}, the finite free multiplicative convolution is also defined for polynomials $p$ and $q$ given in \eqref{eq:polypq} as 
			\[ p(z) \boxtimes_n q(z) := \sum_{k=0}^n z^{n-k} (-1)^{k} \frac{a_k^p a_k^q}{ \binom{n}{k} }. \] 
			This convolution can also be shown to be the expected characteristic polynomial of a random matrix \cite{Marcus-Spielman-Srivastava2022}. The finite free multiplicative convolution was also classically studied in \cite{MR1544526} in a different context.
			
			The identity polynomial for the finite free multiplicative convolution is $q(z) = (z- 1)^n$, which is the characteristic polynomial of the identity matrix. { It was observed in \cite[Section 5.3.4]{Mirabelli2021}, that similar to \eqref{eq:sumderiv}, many differential operators can be implemented by finite free multiplicative convolution:

			}
			
			
			
			\begin{lemma}[Lemma 3.24 from \cite{Mirabelli2021}, see also \cite{Marcus-Spielman-Srivastava2022}]\label{lem:diff and boxtimes}  If $P$ and $Q$ are polynomials such that $p(z) = P(z D) (z-1)^d$ and $q(z) = Q(z D) (z-1)^d$, then
				\[ p(z) \boxtimes_d q(z) = P(z D) Q(z D) (z-1)^d = P(z D) q(z) = Q(z D) p(z). \]
			\end{lemma}


			A combinatorial description of finite free convolutions in terms of the posets of partitions is given in \cite{Arizmendi-Perales2018}, which we briefly describe. Before doing so, we will introduce the necessary combinatorial definitions and notations. A partition, $\pi = \{ V_1, \ldots, V_r \}$ of $[j] := \{ 1,\ldots, j\}$ is a collection of pairwise disjoint, non-null, sets $V_i$ such that $\cup_{i=1}^{r} V_i = [j]$. We refer to $V_i$ as the blocks of $\pi$, and denote the number of blocks of $\pi$ as $|\pi|$. The set of all partitions of $[j]$ is denoted $\mathcal{P}(j)$, and the set of all pair partitions, meaning partitions with $|V_i|=2$ for all $i$, is denoted $\mathcal{P}_2(j)$.
			
			The set of partitions can be equipped with the partial order $\preceq$ of reverse refinement, where we define $\pi \preceq \sigma$ if every block of $\pi$ is completely contained in a block of $\sigma$. The minimal element in this ordering is $0_j : = \{ \{1\}, \{2\}, \ldots, \{j\} \}$ and the maximal element is $1_j = \{\{1,2, \ldots, j \}\}$. The supremum of $\pi$ and $\sigma$ is denoted $\pi \vee \sigma$.  For a partition $\pi =  \{ V_1, \ldots, V_r \}$ and a sequence of numbers $\{c_n\}$, we use 
			\begin{equation} \label{eq:cpinotation}
				c_{\pi}:= \prod_{i=1}^{|\pi|}  c_{|V_i|},  \qquad c_{2\pi}= \prod_{i=1}^{|\pi|} c_{2|V_i|},
			\end{equation}
			and $(n)_{j}$ denotes the Pochhammer sequence, $(n)_j = \prod_{i=1}^j (n-i+1)$. 
			
			The M\"obius function (in the set of partitions) is given by 
			\[ \mu(\sigma,\rho):= (-1)^{|\sigma| - |\rho|} (2!)^{r_3} (3!)^{r_4} \ldots ((n-1)!)^{r_n}, \]
			where $r_i$ is the number of blocks of $\rho$ that contain exactly $i$ blocks of $\sigma$. Note that $\mu(\sigma,\rho) $ is $0$ unless $\sigma\preceq \rho$. In particular, we have:
			\[ \mu( 0_n, \rho) = (-1)^{n - |\rho|} \prod_{V \in \rho} (|V| - 1)! =  (-1)^{n - |\rho|} (2!)^{t_3} (3!)^{t_4} \ldots ((n-1)!)^{t_n},\]
			where $t_i$ is the number of blocks of $\rho$ of size $i$.\\

			Recall that we associate the empirical root distribution, $\mu_p$, to a polynomial $p$, given by:
			\[ \mu_p := \frac{1}{n} \sum_{i=1}^n \delta_{\lambda_i(p)}, \]
			where $\lambda_1(p), \ldots, \lambda_n(p)$ are the roots of $p$, counted with multiplicity. Furthermore, the moments and absolute moments of this measure are given by:
			\[ m_j(p) := \frac{1}{n} \sum_{i=1}^n \lambda_i(p)^j \quad  |m|_j(p) := \frac{1}{n} \sum_{i=1}^n |\lambda_i(p)|^j. \]
			From the Newton identities, the coefficients of $p$ can be recovered from the moments:
			\begin{equation} \label{eq:coeffmoment} a_k^p = \frac{1}{k!} \sum_{\pi \in P(k)} n^{|\pi|} \mu(0_k, \pi) m_\pi(p) 
			\end{equation}
			and
			\begin{equation} \label{eq:momentcoeff} m_j(p) = \frac{(-1)^j}{n(j-1)!} \sum_{\pi \in P(j)} (-1)^{|\pi|} N!_{\pi} (|\pi| -1)! a_{\pi}^p \end{equation}
			for $j \in [n]$, where $N!_{\pi} : = \prod_{V \in \pi} |V|!$.  Here, $m_\pi(p)$ and $a_{\pi}^p$ denote the products introduced in \eqref{eq:cpinotation}.

			In classical probability theory, it is often useful to work with the cumulants of a random variable, $X$, given by the coefficients of the log-moment generating function $K_X(t) = \log( \E[e^{t X}])$ of $X$. In \cite{Arizmendi-Perales2018}, the analogous finite free cumulants are defined in several ways, the most direct way is via \cite[Proposition 3.4]{Arizmendi-Perales2018}
			\begin{equation} \label{eq:cumulantcoeff}
				\kappa_j^n(p) : = \frac{ (-n)^j }{ n (j-1)!} \sum_{\pi \in {P(j)}} (-1)^{|\pi|} \frac{N!_{\pi} a_{\pi}^p ( |\pi| -1 )! }{ (n)_\pi },
			\end{equation}
			where $p$ is given in \eqref{eq:polypq}.  For a non-monic polynomial $p$ with leading coefficient $a$, we define the finite free cumulants of $p$ simply by $\kappa_{j}^{n}(p)=\kappa_j^{n}(a^{-1}p)$ for convenience of notation. 
			They can also be defined as the coefficients of the (truncated) $R$-transform, given in \cite{Marcus2021}. This relationship can be inverted (see \cite[Proposition 3.4]{Arizmendi-Perales2018}) to get:
			\begin{equation} \label{eq:coeffcumulant} 
				a_k^p = \frac{ (n)_k }{ n^k k! } \sum_{\pi \in \mathcal{P}(k)} n^{|\pi|} \mu(0_k, \pi) \kappa_\pi^n(p) 
			\end{equation}
			for $k \in [n]$.   By comparing \eqref{eq:coeffmoment} and \eqref{eq:coeffcumulant} we have: \begin{equation}
				\begin{aligned}
					\kappa_{1}^{n}(p)&=m_1(p)\\
					\kappa_{2}^{n}(p)&=\frac{n}{n-1}\left(m_{2}(p)-m_{1}(p)^2\right),
				\end{aligned}
			\end{equation} i.e., $\kappa_{1}^{n}$ is the mean of the roots and $\kappa_{2}^{n}$ is $\frac{n}{n-1}$ times the variance of the roots. 
			
			An important property of the finite free cumulants is that they linearize the finite free additive convolution, just as the classical cumulants linearize the convolution in classical probability, in the sense that:
			\begin{equation} \label{eq:cumulantadd}
				\kappa_j^n( p \boxplus_n q) =  \kappa_j^n( p )  +  \kappa_j^n(q), \qquad j \in [n]. \end{equation}
			
			 The most important property of finite free cumulants for our purposes is that they are very easy to compute after taking derivatives, as first pointed out in \cite{Arizmendi-Fujie-Perales-Ueda2024}.
			\begin{proposition}[Proposition 3.4 in \cite{Arizmendi-Fujie-Perales-Ueda2024}]\label{prop:cumulant der formula}
				Let $p_n$ be a monic degree $n$ polynomial, and let $p_{\ell,n}=\frac{\ell!}{n!}p_n^{(n-\ell)}$ be the monic version of its $(n-\ell)$-th derivative. Then, \begin{equation}
					\kappa_j^{\ell}\left(p_{\ell,n}\right)=\left(\frac{\ell}{n}\right)^{j-1}\kappa_{j}^{n}\left(p_n\right)
				\end{equation} for any $1\leq j\leq \ell$. 
			\end{proposition} We point out that Proposition \ref{prop:cumulant der formula} can be easily verified from \eqref{eq:cumulantcoeff} using only the power rule. In view of the additivity of cumulants, \eqref{eq:cumulantadd}, and their multilinearity, Proposition \ref{prop:cumulant der formula} also gives the cumulants of $n/\ell$-fractional convolution power, renormalized by $ \ell/n$. In light of this connection, it is not surprising high derivatives of polynomials satisfy a central limit theorem.

			In \cite{Arizmendi-GarzaVargas-Perales2023}, Lemma \ref{lem:diff and boxtimes}, was applied with $q(z) = z^n(z-1)^{d}$, the characteristic polynomial of a projection matrix, to show:
			\begin{equation} \label{eq:convderiv}
				p(z) \boxtimes_{n+d} z^n(z-1)^{d}= \frac{1}{(n+d)_n} z^n D^n p(z) .
			\end{equation} 
			for any $n+d$ degree polynomial $p(z)$, which in turn was used to relate fraction additive free convolution to repeated differentiation. This connection between the derivatives of polynomials and free convolutions, had already been noted in \cite{Steinerberger2020}. In standard free probability, this connection between multiplicative convolution with projection operators and fractional additive free convolution is well known \cite{MR1400060}.

			Combining \eqref{eq:cumulantcoeff} and \eqref{eq:coeffmoment}, or \eqref{eq:momentcoeff} and \eqref{eq:coeffcumulant} gives the following moment-cumulant formulas:
			
			\begin{proposition}[Theorem 4.2 from \cite{Arizmendi-Perales2018}]\label{prop:moment cumulant formula}
				Let $p$ be a monic polynomial of degree $n$. Then,  \begin{equation}\label{eq:cumulant moment form}
					\kappa_{j}^{n}(p)=\frac{(-n)^{j-1}}{(j-1)!}\sum_{\sigma\in\mathcal{P}(j)}n^{|\sigma|}\mu(0_{j},\sigma)m_{\sigma}(p)\sum_{\pi\geq\sigma}\frac{(-1)^{|\pi|-1}(|\pi|-1)!}{(n)_\pi},
				\end{equation} for $j=1,\dots,n$ and \begin{equation}\label{eq:moment cumulant form}
					m_{j}(p)=\frac{(-1)^{j-1}}{n^{j+1}(j-1)!}\sum_{\sigma\in\mathcal{P}(j)}n^{|\sigma|}\mu(0_{j},\sigma)\kappa_{\sigma}^{n}(p)\sum_{\pi\geq\sigma}(-1)^{|\pi|-1}(n)_{\pi}(|\pi|-1)!,
				\end{equation} for $j\in\N$.
			\end{proposition}
			
			We note that \eqref{eq:cumulant moment form} gives a crude bound on the cumulants in terms of the absolute moments:
			\begin{equation} \label{eq:cumulantmomentbound}
				|\kappa_{j}^{n}(p)| \leq C_j n^{-1} \sum_{\sigma \in \mathcal{P}(j)} n^{|\sigma|} |m|_{\sigma}(p) 
			\end{equation}
			for some constant $C_j > 0$.
			
			 We summarize the remaining facts needed for finite free cumulants in the following proposition. \begin{proposition}\label{prop: cumulant facts.}
				Let $\{p_n\}_{n\geq 1}$ and $p$ be monic polynomials of fixed degree $d$. Then, \begin{enumerate}
					\item The finite free cumulants $\kappa_{1}^{d}(p),\dots, \kappa_{d}^{d}(p)$ uniquely determine the polynomial $p$.
					
					\item $p_n(z)\rightarrow p(z)$ uniformly on compact subsets as $n\rightarrow\infty$ if and only if $(\kappa_{1}^{d}(p_n),\dots,\kappa_{d}^{d}(p_n)) \rightarrow (\kappa_{1}^{d}(p),\dots,\kappa_{d}^{d}(p))$  in any norm on $\R^{d}$. 
					
					\item If $p(z)=(z-a)^{d}$, then $\kappa_{j}^{d}(p)=a\delta_{j1}$ for any $1\leq j\leq d$, where $\delta_{ij}$ is the Kronecker delta function. Additionally, $\kappa_{j}^{d}(\He_{d})=d\delta_{j2}$ for any $1\leq j\leq d$. 
				\end{enumerate}
			\end{proposition}
			
			\begin{proof}
				The first claim follows immediately from \eqref{eq:cumulantcoeff} and \eqref{eq:coeffcumulant}.
				
				For the second claim, we first recall that all norms of $\R^d$ are equivalent. For convenience we will choose the $\ell^{2}$ norm.  Note by \eqref{eq:coeffcumulant} that the coefficients are given by polynomials $Q_{d,j}$ in the finite free cumulants, and hence continuous in the finite free cumulants. Thus, $(\kappa_{1}^{d}(p_n),\dots,\kappa_{d}^{d}(p_n)) \rightarrow (\kappa_{1}^{d}(p),\dots,\kappa_{d}^{d}(p))$  in $\ell^2(\R^{d})$ if and only if $(a_{1}^{p_n},\dots,a_{d}^{p_n}) \rightarrow (a_{1}^{p},\dots,a_{d}^{p})$  in $\ell^2(\R^d)$. It is then straightforward to show for any compact set $K\subset\C$, that $\sup_{z\in K}|p_n(z)-p(z)|\rightarrow0$ if and only if $\|(a_{1}^{p_n},\dots,a_{d}^{p_n}) - (a_{1}^{p},\dots,a_{d}^{p})\|_{2}\rightarrow 0$. 
				
				For a monic degree $d$ polynomial $p$, let $\hat{p}$ be any formal power series such that $\hat{p}(D)z^d=p(z)$. \cite[Definition 3.2]{Arizmendi-Perales2018} originally define the finite free cumulants as the coefficients mod $s^{d}$ of the formal power series \begin{equation}
					-\frac{\hat{p}'(sd)}{\hat{p}(sd)}\mod s^{d}=\sum_{k=1}^{d}\kappa_{k}^{d}(p)s^{k-1}.
				\end{equation} 
				
				The third claim can be verified after noting $(z-a)^{d}=e^{-aD}z^d$ and $\He_{d}(z)=e^{-\frac{D^2}{2}}z^d$. 
			\end{proof}

			We conclude this section with a discussion of limit theorems in finite free probability. In \cite[Section 6]{Marcus2021} and \cite[Section 6]{Arizmendi-Perales2018} limit theorems are given for a growing number of convolutions of fixed degree polynomials. These limits are analogous to the well known Law of Large Numbers and Central Limit Theorem in classical probability theory, which roughly state that the classical convolution-powers of a measure converge to the mean of the measure or to a Gaussian distribution, depending on their rescaling. Among other similar results, \cite{Marcus2021} and \cite{Arizmendi-Perales2018} show that if $\{p_i\}_{i \in \N}$ is a sequence of degree $d$ polynomials each with $m_1(p_i) =0$ and $m_2(p_i) =1$, then the $n$-fold finite free convolution, $n^{-nd/2} p_1(\sqrt{n} z) \boxplus_d \ldots \boxplus_d p_n(\sqrt{n} z) $ converges to the $d^{th}$ Hermite polynomial as $n \to \infty$.  
			
			In other words, the Hermite polynomials play the role of the Gaussian distribution in finite free probability, in particular all of the (finite free) cumulants after the second vanish for both distributions. The emergence of Hermite polynomials in the finite free central limit theorem and in Theorem \ref{thm:CLT} does not seem to be coincidental. 
			

			
			\section{Proofs of the theorems in Section \ref{sec: main results uni principles}}\label{sec:proofsuni} Throughout this section, we will assume $f$ satisfies Assumption \ref{assump:root density}. \Cr Specifically, the following lemma, {\cedittwo which we prove in Section \ref{sec:proof_of_coefficients},} is available throughout. \begin{lemma}\label{lem:coefficients from roots}
				Let $f$ be an even entire function, as represented in \eqref{eq:even series representation}, which satisfies Assumption \ref{assump:root density}. Then, {\cedittwo with $a_n$ as in  \eqref{eq:defan}},
				\begin{equation}\label{eq:coeff ratio}
					-\frac{\gamma_{2n}}{\gamma_{2(n-1)}}=a_n^{-1}(1+o(1)),
				\end{equation} and \begin{equation}\label{eq:ratio of coeff ratio condition}
					\lim_{n\rightarrow\infty}\frac{\gamma_{2(n-2)}\gamma_{2n}}{\gamma_{2(n-1)}^2}=1.
				\end{equation} 
			\end{lemma} \nc We will prove Theorem \ref{thm:Cosine uni}-\ref{thm:Laguerre uni} by applying the results of Section \ref{sec:finite free limit theorems}. 
			We begin by specializing the formulas for finite free cumulants, \eqref{eq:cumulantcoeff}, to Jensen polynomials: 
			\begin{equation}\label{eq:even Jensen cumulant formula}
				\kappa_{j}^m(J_{m,0})=\frac{m^{j-1}}{(j-1)!}\sum_{\pi\in\mathcal{P}(j)} (-1)^{|\pi|}(|\pi|-1)!\frac{(2m)_{2\pi}}{(m)_\pi}\frac{\gamma_{2(m-\pi)} }{\gamma_{2m}^{|\pi|}}.
			\end{equation}
			In particular, for $j = 1,2$, we have:
			\begin{equation}\label{eq:even Jensen cumulant first two}
				\begin{aligned}
					\kappa_{1}^{m}\left(J_{m,0} \right)&=-2(2m-1)\frac{\gamma_{2(m-1)}}{\gamma_{2m}},\\
					\kappa_{2}^{m}\left(J_{m,0} \right)&=4(2m-1)^2m\left[\left(\frac{\gamma_{2(m-1)}}{\gamma_{2m}}\right)^2-\frac{2m-3}{2m-1}\frac{\gamma_{2(m-2) }}{\gamma_{2m}} \right].
				\end{aligned}
			\end{equation} 
			%
			Additionally, we can essentially factor out $\left(\kappa_{1}^{m}\right)^{j}$ from \eqref{eq:even Jensen cumulant formula} to make the relative growth of the finite free cumulants more transparent: 
			\begin{equation}\label{eq:even cumulant asympt}
				\begin{aligned}
					\kappa_{j}^{m}\left(J_{m,0} \right)=\frac{m^{j-1}2^{j}(2m-1)^{j}}{(j-1)!}\left(\frac{\gamma_{2(m-1)}}{\gamma_{2m}} \right)^{j}\sum_{\pi\in\mathcal{P}(j)} \frac{(-1)^{|\pi|}(|\pi|-1)!}{2^{j}(2m-1)^{j}}\frac{(2m)_{2\pi}}{(m)_\pi}\frac{\gamma_{2(m-\pi)} }{\gamma_{2m}^{|\pi|-j}\gamma_{2(m-1)}^{j}}.
				\end{aligned}
			\end{equation}
			We will primarily use \eqref{eq:ratio of coeff ratio condition} to simplify the summation on the right-hand side of \eqref{eq:even cumulant asympt}. Recently, \cite{Arizmendi-Fujie-Perales-Ueda2024} defined an $S$-transform for finite free probability by using the ratio of the coefficients of a polynomial. In the language of \cite{Arizmendi-Fujie-Perales-Ueda2024}, \eqref{eq:coeff ratio} and \eqref{eq:ratio of coeff ratio condition} can be translated into statements on the finite free $S$-transforms of the even Jensen polynomials evaluated near $0$. 
			

			\subsection{Proof of Theorem \ref{thm:Laguerre uni}} 
			
			We will prove Theorem \ref{thm:Laguerre uni} by applying Theorem \ref{thm:Poisson limit theorem}. 
			
			\begin{proof}[Proof of Theorem \ref{thm:Laguerre uni}]
				We begin by showing that after rescaling the argument of $W_{n+d,0} = J_{n+d,0}$ by $a_n$, the first 2 cumulants satisfy the appropriate assumptions.
				After rescaling the argument of \eqref{eq:even Jensen cumulant first two} by $a_n$, we have
				\[ \kappa_{1}^{n+d}\left(W_{n+d,0}(a_n z) \right)=-2(2(n+d)-1)\frac{\gamma_{2(n+d-1)}}{\gamma_{2(n+d)}} a_n^{-1} . \]
				So by the definition of $a_n$ in \eqref{eq:defan}, we have 
				\[  \kappa_{1}^{n+d}\left(W_{n+d,0}(a_n z) \right)/n \sim 4 .\]
				Furthermore, rearranging the expression for $\kappa_2$ as in \eqref{eq:even cumulant asympt}, gives
				\[ \kappa_{2}^{n+d}(W_{n+d,0}) = (n+d) \kappa_{1}^{n+d}(W_{n+d,0})^2  \left[  1 -\frac{2(n+d)-3}{2(n+d)-1}\frac{\gamma_{2(n+d-2)}  \gamma_{2(n+d)} }{\gamma_{2(n+d-1)}^2 }                \right]
				\]	 
				Then, by taking $n\to\infty$ and applying Lemma \ref{lem:coefficients from roots}, we see that $\kappa_{2}^{n+d}(W_{n+d,0})=o((n+d)\kappa_{1}^{n+d}(W_{n+d,0})^{2})$. In particular we have that $\kappa_{2}^{n+d}\left(W_{n+d,0}(a_n z) \right) = o (n^3)$, as required in Theorem \ref{thm:Poisson limit theorem}.

				
				Recall the definition of $M$ given in \eqref{eq:M definition}. In order to apply Theorem \ref{thm:Poisson limit theorem}, we first rewrite $\frac{1}{\gamma_{2n}}W_{d,n}(a_nz)$ as the application of $M^n$ to a sequence of monic polynomials:
				\begin{equation}
					\begin{aligned}
						\frac{1}{\gamma_{2n}}W_{d,n}(a_nz)&=\frac{1}{\gamma_{2n}}\frac{a_{n}^{-n}}{(n+d)_{n}}M^{n}W_{d+n,0}(a_nz)\\
						&=\frac{1}{\gamma_{2n}}\frac{d!}{(2d)!}\frac{(2n+2d)!}{(n+d)!}\frac{a_{n}^{-n}}{(2n+2d)_{2n}}M^{n}W_{d+n,0}(a_nz)\\
						&=\frac{\gamma_{2(n+d)}}{\gamma_{2n}}\frac{d!}{(2d)!} \frac{a_{n}^{d}}{(2n+2d)_{2n}}M^{n} \frac{(2n+2d)!}{(n+d)!\gamma_{2(n+d)} a_{n}^{n+d}}  W_{d+n,0}(a_nz), 
					\end{aligned}
				\end{equation} 
				where in the final line we have multiplied  $W_{d+n,0}(a_nz)$ by the appropriate term to make it monic.
				By the Lemma \ref{lem:coefficients from roots} we have
				\begin{equation}
					\lim_{n\rightarrow\infty}\frac{\gamma_{2n+2d}}{\gamma_{2n}}a_n^{d}=(-1)^d. 
				\end{equation} 
				Hence, Theorem \ref{thm:Laguerre uni} follows by applying Theorem \ref{thm:Poisson limit theorem}, and rescaling the argument by $1/4$ to account for the asymptotic $\kappa_{1}^{n+d}\left(J_{n+d,0}(a_n z) \right) \sim 4 n$. 
			\end{proof}
			\subsection{Proof of Theorem \ref{thm:Hermite uni}} To prove Theorem \ref{thm:Hermite uni}, we will apply Theorem \ref{thm:CLT} to the polynomials $J_{d+n,0}$. 
			
			\begin{proof}[Proof of Theorem \ref{thm:Hermite uni}]
				As can be seen directly from \eqref{eq:even Jensen cumulant first two} $b_n$ and $c_n$ are exactly the first and second finite free cumulants of $J_{d+n,0}$ respectively. So to apply Theorem \ref{thm:CLT} it suffices to show $|m|_{3}(J_{d+n,0}) = o(\sqrt{(n+d)c_n^3})$. 
				Furthermore, since the roots of $J_{n+d,0}$ are non-negative, we have 
				$ |m|_{3}(J_{d+n,0}) = m_{3}(J_{n+d}) $. 
			
			Notice, from Proposition \ref{prop:moment cumulant formula} that the leading order of $m_{3}$ is in fact $\kappa_{3}^{n+d}(J_{n+d})$ and it is in fact enough to prove the same bound for the third cumulant. Let  $r(m)$ be such that \begin{equation}
				\frac{\gamma_{2(m-2)}\gamma_{2m}}{\gamma_{2(m-1)}^2}=1+r(m).
			\end{equation} We will assume $|r(m)|\gg \frac{1}{m}$, otherwise a similar proof to the following is also available. One can check directly from \eqref{eq:even cumulant asympt} that 
			\begin{equation}
				\begin{aligned}
					\frac{\kappa_{3}^{n+d}(J_{n+d,0})}{c_{n}^{3/2}}&=\frac{\sqrt{n+d}}{2}\left[\frac{2+\frac{(2(n+d)-3)(2(n+d)-5)}{(2(n+d)-1)^2}\frac{\gamma_{2(n+d-3)}\gamma_{2(n+d)}^2}{\gamma_{2(n+d-1)}^3}-3\frac{2(n+d)-3}{2(n+d)-1}\frac{\gamma_{2(n+d)}\gamma_{2(n+d-2)} }{\gamma_{2(n+d-1)}^2} }{1-\frac{2(n+d)-3}{2(n+d)-1}\frac{\gamma_{2(n+d)}\gamma_{2(n+d-2)}}{\gamma_{2(n+d-1)}^2 } }\right]\\
					&\lesssim \sqrt{n+d}\frac{|r(n+d)|^2+o(|r(n+d)|^2) }{[|r(n+d)|+o(|r(n+d)|)]^{3/2}}=o(\sqrt{n+d}).
				\end{aligned}
			\end{equation} 
			Thus, we can apply Theorem \ref{thm:CLT} to complete the proof of Theorem \ref{thm:Hermite uni}.
		\end{proof}

		\subsection{Proof of Theorem \ref{thm:Cosine uni}} We will prove Theorem \ref{thm:Cosine uni} by showing the power series for rescaled $2n^{th}$ derivative of $f$ converges to the power series of cosine.
		
		\begin{proof}[Proof of Theorem \ref{thm:Cosine uni}]
			From Proposition \ref{prop:cosine as finite free poisson} and Theorem \ref{thm:Laguerre uni} we know that the even Jensen polynomials of $\frac{1}{\gamma_{2n}}f^{(2n)}(\sqrt{a_n}z)$ converge to the even Jensen polynomials of $\cos(z)$. Thus, for any fixed $d\in\N$ the first $2d$ coefficients in the series expansion of $\frac{1}{\gamma_{2n}}f^{(2n)}(\sqrt{a_n}z)$ converge to those of $\cos(z)$. We now show the remaining coefficients are negligible. Fix  a compact subset $K\subset\C$ and $\eps>0$.  For $n\in\N$, define \begin{equation}
				\widetilde{f}_{n}(z)=\frac{1}{\gamma_{2n}}f^{(2n)}(\sqrt{a_n}z) = \sum_{k=0}^\infty \frac{\gamma_{2(k+n)}}{\gamma_{2n} (2k)!} a_n^k z^{2k}.
			\end{equation} We let $b_{k,n}: =  \frac{\gamma_{2(k+n)}}{\gamma_{2n} (2k)!} a_n^k$ so we have:\begin{equation}
				\widetilde{f}_{n}(z)=\sum_{k=0}^{\infty}b_{k,n}z^{2k}.
			\end{equation} \Cr At this point we recall that $n_{+}$ is regularly varying by Assumption \ref{assump:root density}, and hence  $a_{n}$ is regularly varying, see \cite[Proposition 0.8]{Resnick87}, and $a_{n}=\widetilde{h}(n)n^{-2+\frac{2}{\alpha}}$ for some \emph{slowly varying} function $\widetilde{h}$\nc. From the Karamata representation theorem, see \cite[Theorem 0.6]{Resnick87}, there exists $c,w:(0,\infty)\rightarrow[0,\infty)$ such that $c(x)\rightarrow c\in(0,\infty)$ as $x\rightarrow\infty$, $w(x)\rightarrow0$ as $x\rightarrow\infty$, and \begin{equation}\label{eq:Kara}
				\widetilde{h}(x)=c(x)\exp\left(\int_{1}^{x}\frac{w(t)}{t}dt\right).
			\end{equation}  We additionally assume $n$ is sufficiently large such that the $1+o(1)$ term in the right-hand side of \eqref{eq:coeff ratio} is in $(1/2,2)$, that $\sup_{x\geq0}c(n+x)\geq c/2$, and that $\sup_{x\geq0} w(n+x)\leq \frac{1}{2}$.
			
			It follows from \eqref{eq:Kara} and the definition of $a_n$ that, for any $k=0,1,\dots$,
			\begin{equation}
				\begin{aligned}
					\left|\frac{\gamma_{2(n+k)} }{\gamma_{2n} } a_n^k \right|&=  \prod_{i=0}^{k-1} \left|\frac{\gamma_{2(n+i+1)} }{\gamma_{2(n+i)} } a_n \right|\\
					&\lesssim 2^{\left|2-\frac{2}{\alpha} \right|k }\prod_{i=0}^{k-1}\frac{\widetilde{h}(n)}{\widetilde{h}(n+i)}\left(1+\frac{i}{n}\right)^{2-\frac{2}{\alpha}}\\
					&\lesssim 2^{\left|2-\frac{2}{\alpha} \right|k+k }c^{-k}k^{\frac{k}{2}}\prod_{i=0}^{k-1}\left(1+\frac{i}{n}\right)^{2-\frac{2}{\alpha}}\\
					&\lesssim 2^{\left|2-\frac{2}{\alpha} \right|k+k }c^{-k}k^{\frac{3k}{2}}.
				\end{aligned}
			\end{equation} 
				Hence, for $n$ sufficiently  large\begin{equation}
					\left|b_{k,n}\right| \lesssim \frac{k^{c'k}}{(2k)!},
				\end{equation} uniformly in $k\in\N$ for some $c'\in (0,2)$.
				

				
					
					Thus, there exists $d\in\N$ such that for $n$ sufficiently large \begin{equation} 
						\sum_{k=2d+2}^{\infty}|b_{k,n}|C^{2k}<\eps,
					\end{equation} 
					where $C$ is chosen sufficiently large that $|z| < C$ for all $z \in K$.
					We then conclude that \begin{equation}
						\begin{aligned}
							\limsup_{n\rightarrow\infty} \sup_{z\in K}|\widetilde{f}_{n}(z)-\cos(z)|&< \limsup_{n\rightarrow\infty} \sum_{k=0}^{2d}\left|b_{k,n}-\frac{(-1)^k}{(2k)!} \right|C^{2k}+\eps\\
							&=\eps.
						\end{aligned}
					\end{equation} As $\eps>0$ was arbitrary, the proof is complete.
				\end{proof}
				
				\section{Proofs of the theorems in Section \ref{sec:finite free limit theorems}}\label{sec:limit theorems proofs} 
				
				We begin by fixing $d\in\N$ and letting $D$ be the derivative operator.  From Proposition \ref{prop:cumulant der formula} we note that \begin{equation}\label{eq:applied cumulant der formula LLN}
					\kappa_{j}^{d}\left(\frac{1}{(n+d)_{n}}D^{n}P_{n+d}\right)=\left(\frac{d}{n+d}\right)^{j-1}\kappa_{j}^{n+d}\left(P_{n+d}\right),
				\end{equation} for any $1\leq j\leq d$. Additionally, \begin{equation}\label{eq:applied cumulant der formula CLT}
					\kappa_{j}^{d}\left(\frac{1}{(n+d)_{n}}D^{n}\DD_{\sqrt{n+d}}P_{n+d}\right)=\left(\frac{d}{n+d}\right)^{j-1}(n+d)^{j/2}\kappa_{j}^{n+d}\left(P_{n+d}\right),
				\end{equation} for any $1\leq j\leq d$. In what follows we will use that, for any $1\leq r\leq k<\infty$, we have the bound:
				\begin{equation}\label{eq:higher moment bounds}
					|m_k|\left(P_{n+d}\right)\leq (n+d)^{\frac{k}{r}-1}|m_{r}|\left(P_{n+d}\right)^{k/r},
				\end{equation} which follows from convexity or standard bounds on $\ell^{p}(\R^{n+d})$-norms.
				
				\begin{proof}[Proof of Theorem \ref{thm:Law of large numbers} ] 
					It follows from \eqref{eq:cumulantmomentbound}, the assumption that $m_{2}(P_{n+d})=o(n+d)$, and \eqref{eq:higher moment bounds} that \begin{equation}\label{eq:LLN cumulant bound}
						\left|\kappa_{j}^{n+d}(P_{n+d})\right|=o\left((n+d)^{j-1} \right),
					\end{equation} for any $j\geq 2$. Fix $1\leq j\leq d$. It then follows from \eqref{eq:LLN cumulant bound} and \eqref{eq:applied cumulant der formula LLN} that \begin{equation}
						\lim\limits_{n\rightarrow\infty}	\kappa_{j}^{d}\left(\frac{1}{(n+d)_{n}}D^{n}P_{n+d}\right)=a\delta_{j1},
					\end{equation} where $\delta_{ij}$ is the Kronecker delta function. The proof is then complete by applying Proposition \ref{prop: cumulant facts.}.
				\end{proof}
				
				\begin{proof}[Proof of Theorem \ref{thm:CLT}]
					Fix $1\leq j\leq d$. Using \eqref{eq:cumulantmomentbound}, \eqref{eq:applied cumulant der formula CLT}, and \eqref{eq:higher moment bounds} as in the proof of Theorem \ref{thm:Law of large numbers}, but instead using the assumption $m_2(P_m) \to 1$, we see that \begin{equation}
						\lim\limits_{n\rightarrow\infty}	\kappa_{j}^{d}\left(\frac{1}{(n+d)_{n}}\DD_{\sqrt{n+d}}D^{n}P_{n+d}\right)=d\delta_{j2},
					\end{equation} where $\delta_{ij}$ is the Kronecker delta function. The proof is then complete by applying Proposition \ref{prop: cumulant facts.}.
				\end{proof}

					\begin{proof}[Proof of Corollary \ref{cor:limits of Wigner characteristic functions}]
						It follows from well known results in random matrix theory  (see for example Theorems 2.3.24 and 2.4.2 in \cite{Tao2012})  that almost surely
						\begin{equation}
							\frac{\kappa_{2}^{n+d}(\Phi_{n+d})}{n}\sim 1,\quad \kappa_{1}^{n+p}(\Phi_{n+d})=o\left(1 \right),\text{ and}\quad m_{4}^{n+d}(\Phi_{n+d})=O(n^2).
						\end{equation} Hence, shifting and normalizing $\Phi_{n+d}$ to have mean $0$ and second cumulant $1$ is undone by the rescaling in Theorem \ref{thm:CLT}. So we may apply Theorem \ref{thm:CLT} directly to $\Phi_{n+d}$.
					\end{proof}

					\subsection{Poisson limit theorem and the proof of Theorem \ref{thm:Poisson limit theorem}} In this section we prove Theorem \ref{thm:Poisson limit theorem} as corollary of Theorem \ref{thm:Law of large numbers}. An alternative proof is available by considering $D^{2n}P_{n+d}(z^2)$, applying Theorem \ref{thm:CLT}, and using the relationship between Hermite and Laguerre polynomials. However, we choose the below proof as it can be generalized under straightforward modifications to repeated application of the more general operators discussed in Remark  \ref{rem:General laguerre}. For clarity of presentation we consider only the operator $M$.
					
					\begin{proof}[Proof of Theorem \ref{thm:Poisson limit theorem}]
						We define a new array of polynomials $\{W_{d,n}\}_{d,n=1}^\infty$ by \begin{equation}
							W_{d,n}(z)=\frac{1}{(2n+2d)_{2n}}M^nW_{n+d,0}(z),\quad W_{m,0}(z)=P_{m}(z), 
						\end{equation} where $M$ is defined by \eqref{eq:M definition}.
						
						We then factorize $M^n$ into the product of $D^n$ and a differential operator that only depends on $z$ and $D$ through $zD$:
						\begin{equation}\label{eq:M decomposition}
							\frac{1}{(2n+2d)_{2n}}M^n=\frac{(n+d)_{n}}{(2n+2d)_{2n}}\mathcal{F}_n\frac{D^n}{(n+d)_n},
						\end{equation} where \begin{equation}\label{eq:Fn definition}
							\mathcal{F}_n=\prod_{k=0}^{n-1}\left(2k+[1+2zD ] \right), 
						\end{equation} which is a polynomial in $zD$. As an operator on polynomials, $\mathcal{F}_n$ is degree preserving. Then, by Lemma \ref{lem:diff and boxtimes} \begin{equation} \label{eq:MnP}
							\begin{aligned}
								\frac{1}{(2n+2d)_{2n}}M^n P_{n+d}(z)&=\frac{(n+d)_{n}}{(2n+2d)_{2n}}\mathcal{F}_n\frac{D^n}{(n+d)_n} P_{n+d}(z)\\
								&=\frac{D^n}{(n+d)_n} P_{n+d}(z)\boxtimes_{d} \frac{(n+d)_{n}}{(2n+2d)_{2n}}\mathcal{F}_n(z-1)^d\\
								&=\frac{D^n}{(n+d)_n}n^{-n-d} P_{n+d}(nz)\boxtimes_{d} \frac{(n+d)_{n}}{(2n+2d)_{2n}}\DD_{n}\mathcal{F}_n(z-1)^d.
							\end{aligned}
						\end{equation}  In the last term we simply moved some scaling from $P_{n+d}$ to $(z-1)^d$ using that scalars can be moved across multiplicative convolutions. We know from Theorem \ref{thm:Law of large numbers} that \begin{equation}
							\lim\limits_{n\rightarrow\infty} \frac{D^n}{(n+d)_n}n^{-n-d} P_{n+d}(nz)=(z-a)^d.
						\end{equation} 
						In the basis $\left\{z^{0},z^{1},\dots,z^{d} \right\}$ $\mathcal{F}_n$ is a diagonal matrix:
						\begin{equation} 
							\mathcal{F}_n z^j =2^{n}\prod_{k=0}^{n-1}\left(k+\frac{1}{2}+j \right) z^j 
						\end{equation} for $0\leq j\leq d$. 
						For each $j\in [d]$, we factorize the coefficient as: \begin{equation*}
							\begin{aligned} 
								\prod_{k=0}^{n-1}\left(k+\frac{1}{2}+j \right)  
								&=  \left(n+j-\frac{1}{2} \right)_{j} \left(n-\frac{1}{2} \right)_{n-d} \left(d-\frac{1}{2} \right)_{d-j} .
							\end{aligned}
						\end{equation*}
						Hence \begin{equation*}		
							\begin{aligned}
								\mathcal{F}_n(z-1)^{d}&=2^{n} \sum_{j=0}^{d}\binom{d}{j}(-1)^{d-j} \left( \left(n+j-\frac{1}{2} \right)_{j} \left(n-\frac{1}{2} \right)_{n-d} \left(d-\frac{1}{2} \right)_{d-j} \right) z^j \\
								&=2^{n}d!(-1)^{-d}\left(n-\frac{1}{2}\right)_{n-d}\sum_{j=0}^{d}\frac{\left(d-\frac{1}{2} \right)_{d-j}}{j!(d-j)!} \left(n+j-\frac{1}{2} \right)_{j}  (-1)^jz^j. 
							\end{aligned}
						\end{equation*} 
						
						We then rescale the argument by $n$ to get:
						\begin{equation}\label{eq:F_n limit with scaling computation}
							\begin{aligned}
								\DD_{n} \mathcal{F}_n(z-1)^{d}	&=2^{n}d!(-1)^{-d}\left(n-\frac{1}{2}\right)_{n-d}{n}^d \sum_{j=0}^{d}\frac{\left(d-\frac{1}{2} \right)_{d-j}}{j!(d-j)!}\frac{\left(n+j-\frac{1}{2} \right)_{j}}{n^j}(-1)^jz^j. 
							\end{aligned}
						\end{equation}

						%
						Since $\left(n+j-\frac{1}{2} \right)_{j}$ is a monic degree $j$ polynomial in $n$, we have that $\frac{\left(n+j-\frac{1}{2} \right)_{j}}{n^{j}}\rightarrow 1$ as $n\rightarrow\infty$. We note that the $n$-dependent coefficient on the right most side of \eqref{eq:F_n limit with scaling computation} exactly cancels with the scaling term in \eqref{eq:MnP}:
						\begin{equation}
							\lim\limits_{n\rightarrow\infty}\frac{(n+d)_{n}}{(2n+2d)_{2n}}2^{n}n^d\left(n-\frac{1}{2}\right)_{n-d}=1.
						\end{equation} 
						So $\frac{(n+d)_{n}}{(2n+2d)_{2n}}\DD_{n}\mathcal{F}_n(z-1)^d$ converges to the monic polynomial: 
						\[ d!(-1)^{-d} \sum_{j=0}^{d}\frac{ (-1)^{-d} \left(d-\frac{1}{2} \right)_{d-j}}{j!(d-j)!}z^j  =  d!  (-1)^{-d}L_{d}^{\left(-\frac{1}{2}\right)}(z). \]
						Hence, \begin{equation}
							\lim\limits_{n\rightarrow\infty} W_{d,n}(z)=(z-a)^{d}\boxtimes_{d} d!(-1)^{-d}L_{d}^{\left(-\frac{1}{2}\right)}(z).
						\end{equation} Normalizing such that $a=1$ completes the proof.
					\end{proof}
					
					\section{Proof of Lemma \ref{lem:coefficients from roots} and results in complex analysis} \label{sec:proof_of_coefficients} In this section we collect some results in complex analysis and prove Lemma \ref{lem:coefficients from roots}. First, we note that \eqref{eq:ratio of coeff ratio condition} follows from \eqref{eq:coeff ratio}, and hence we focus only on the latter. 
					
					If $f(0)$ is any non-zero constant we note that the ratios of coefficients of $f$ and $\frac{1}{f(0)}f$ are identical. If $f$ has a zero at the origin of order $2\ell$, then the coefficients of $f$ and $z^{-2\ell}f$ differ by shifting the index $n$ to $n+\ell$. The resulting difference in the ratios of coefficients can be absorbed into the $1+o(1)$ term on the right hand side of \eqref{eq:coeff ratio}. Thus, for simplicity we assume without loss of generality in this section that $f(0)=1$ and let $g(z)=f(\sqrt{z})$.  We express $g$ as the product and define the series coefficients $e_{m}$ by \begin{equation}\label{eq:g product and series}
						g(z)=\prod_{k=1}^{\infty}\left(1-\frac{z}{x_{k}^2}\right)=\sum_{k=0}^{\infty}(-1)^{k}e_{k}z^{k}.
					\end{equation} For simplicity, we denote the square of the roots by $r_{k}=x_{k}^{2}$ and let $n(r)=|\{k\in\N:r_{k}\in [0,r] \}|$. From Assumption \ref{assump:root density} $n(r)\sim \widehat{h}(r)r^{\alpha/2}$ for some positive slowly varying function $\widehat{h}$. We define $\rho=\frac{\alpha}{2}$. Before we proceed we note that standard computations yield the following integrals, which we will use in the limits below. 
					\begin{lemma} \label{lem:calc_integrals}
						Let $\rho \in (0, 1)$.  Then,
						\begin{align} \label{eq:calcint1}
							\int_0^\infty \frac{u^\rho}{(1 + u)^2} \,du &= \rho \pi \csc(\pi \rho), \\
							\int_0^\infty \frac{u^\rho}{(1 + u)^3} \,du &= \frac{1}{2} \rho (1-\rho) \pi \csc(\pi \rho),  \label{eq:calcint2}
						\end{align}
						and
						\begin{equation} \label{eq:calcint3}
							\int_0^\infty \frac{u^\rho}{u(1 + u)} \,du = \pi \csc(\pi \rho). 
						\end{equation}
					\end{lemma}
					
						
						\subsection{Proof of Lemma \ref{lem:coefficients from roots}} We use a saddle point argument similar to that of \cite{Pemantle-Subramanian2017}. 
						
						\begin{proof} Beginning from Cauchy's integral formula \begin{equation}\label{eq:Cauchy form}
								e_{k}=\frac{(-1)^k}{2\pi i}\int_{\Gamma} z^{-k}g(z)\frac{1}{z} dz,
							\end{equation} where $\Gamma$ is a circle centered at the origin with radius to be chosen later. We define the function $\phi_{k}(z)=\log g(z)-k\log z$, where $\log$ is the logarithm with the branch cut along the negative imaginary axis. Our goal is to find a saddle point $\sigma_{k}$ of $\phi_{k}$ and a circle $\Gamma$ centered at the origin passing through $\sigma_{k}$ such that \begin{enumerate}
								\item $\phi_{k}'(\sigma_{k})=0$ and $\phi_{k}''(\sigma_{k})>0$. 
								\item The contribution of the arc of $\Gamma$ of length $\phi_{k}''(\sigma_{k})^{-1/2}$ centered at $\sigma_k$ is $\exp\left(\phi_{k}(\sigma_{k})\right)\sqrt{2\pi/\phi_{k}''(\sigma_{k})}$. 
								
								\item The contributions from the remainder of $\Gamma$ is small.
							\end{enumerate} 
							
								
								The major simplification available compared to \cite{Pemantle-Subramanian2017} is that every sum we will consider converges absolutely and the saddle point will sit exactly on the negative real line. The derivative of $\phi_{k}$ is given by \begin{equation}\label{eq:phi' def}
									\phi_{k}'(z)=-\frac{k}{z}+\sum_{k=1}^{\infty}\frac{1}{z-r_k}=:-\frac{k}{z}+s(z).
								\end{equation} Thus, we are looking for $z\in\R$ such that \begin{equation}\label{eq:sigma goal}
									zs(z)=k,
								\end{equation} and $-z\rightarrow+\infty$.  Define the measure $\Pi_{g}=\sum_{k=1}^{\infty}\delta_{r_k}$, where $\delta_{r_k}$ is the point mass at $r_k$, and throughout assume $z=-r$ and large $r>0$.  We will use the asymptotic notation $\sim$ under the assumption that $r\rightarrow\infty$ or $k\rightarrow \infty$; we distinguish between these two limits with $\sim_{r}$ and $\sim_{k}$, respectively.  Integration by parts gives\begin{equation}\label{eq:integral example}
									\begin{aligned}
										-s(z)&=\int_{0}^{\infty} \frac{1}{r+t}d\Pi_{g}(t)\\
										&=\int_{0}^{\infty} r^{-2}\frac{1}{(1+t/r)^2}n(t) dt\\
										&=\int_{0}^{\infty}r^{-1}\frac{n(ru)}{(1+u)^2} du\\
										&\sim_{r} \left(\int_{0}^{\infty}\frac{u^{\rho}}{(1+u)^2}du \right)r^{-1+\rho}\widehat{h}(r).
									\end{aligned}
								\end{equation}  Hence, by \eqref{eq:calcint1}, $-rs(-r)\sim_{r} \pi\rho\csc(\pi\rho) r^{\rho}\widehat{h}(r)\sim_{r}\pi\rho\csc(\pi\rho)n(r)$, and by the intermediate value theorem,   \eqref{eq:sigma goal} has at least one solution $\sigma_{k}$ on the negative real line such that \Cr \begin{equation}\label{eq:sigma asy}
									\sigma_{k}\sim_{k} -\left(\frac{k}{c_{\rho}\widehat{h}(-\sigma_{k})}\right)^{1/\rho}=-\left(\frac{k}{\pi\rho\csc(\pi\rho)\widehat{h}(-\sigma_{k})}\right)^{1/\rho}\sim_{k} -\left[n_{+}^{-1}\left(\frac{k}{\pi\rho\csc(\pi\rho)}\right)\right]^{2},
								\end{equation} where the last line follows from noting that $n(r)=n_{+}(\sqrt{r})$. \nc For our purposes the uniqueness of $\sigma_{k}$ is immaterial, so we take any choice satisfying \eqref{eq:sigma asy}.  The higher derivatives of $s^{(j)}$ are \begin{equation}
									s^{(j)}(z)=(-1)^{j}(j)!\sum_{k=1}^{\infty} \frac{1}{(z-r_k)^{j+1}}.
								\end{equation} Following similarly to \eqref{eq:integral example} \begin{equation}\label{eq:s higher der}
									s^{(j)}(-r)\sim_{r} -(j+1)!\left(\int_{0}^{\infty}\frac{u^{\rho}}{(1+u)^{j+2}} \right)r^{-j-1+\rho}\widehat{h}(r).
								\end{equation} 
								
								We take $\Gamma$ in \eqref{eq:Cauchy form} to be a circle of radius $|\sigma_{k}|$ centered at the origin. Let $\Gamma_1=\{z:z=\sigma_{k}e^{i\theta},\ -k^{-\delta}\leq \theta\leq k^{-\delta} \}$ be a small arc around $\sigma_{k}$ for some fixed $\delta\in(1/3,1/2)$.  Define \begin{equation}
									v_k(\theta)=\phi_{k}(\sigma_{k}e^{i\theta}).
								\end{equation} It is straightforward to check using \eqref{eq:s higher der} that \begin{equation}\label{eq:v asymp}
									\begin{aligned}
										v'_k(0)=0,\quad
										v''_k(0)\sim_{k} k\left[-1+\frac{2\int_{0}^{\infty}\frac{u^\rho}{(1+u)^3}du}{\int_{0}^{\infty}\frac{u^{\rho}}{(1+u)^2}du } \right],\text{ and }
										\sup_{|\theta|\leq k^{-\delta} }	v'''_{k}(\theta)=O(k).	
									\end{aligned}
								\end{equation} From Lemma \ref{lem:calc_integrals}, it follows that \begin{equation}
									-1+\frac{2\int_{0}^{\infty}\frac{u^\rho}{(1+u)^3}du}{\int_{0}^{\infty}\frac{u^{\rho}}{(1+u)^2}du }=-1+\frac{\pi\rho(1-\rho)\csc(\pi\rho)}{\pi\rho\csc(\pi\rho)}=-\rho,
								\end{equation} and \begin{equation}\label{eq:v'' exact}
									v_{k}''(0)\sim_{k} -\rho k.
								\end{equation} \begin{lemma} We have
									\begin{equation}
										\frac{\int_{\Gamma_1}\frac{g(z)}{z^{k+1}}dz}{g(\sigma_{k})\sigma_{k}^{-k} (\rho k)^{-1/2}}\rightarrow i\sqrt{2\pi}
									\end{equation}
									as $k \to \infty$.
								\end{lemma} 
								
								\begin{proof}
									Following an approach similar to \cite{Pemantle-Subramanian2017}, we parametrize $\Gamma_1$ in terms of the angle $\theta$ and Taylor expand $v_{k}$ up to third order using \eqref{eq:v asymp} to see that \begin{equation}
										\int_{\Gamma_1}\frac{g(z)}{z^{k+1}}dz=i\int_{-k^{-\delta}}^{k^{-\delta}}\exp\left(v_{k}(0)+v_{k}'(0)\theta+v_{k}''(0)\frac{\theta^2}{2} \right)d\theta(1+o(1)),
									\end{equation} where $\exp\left(v_{k}(0) \right)=g(\sigma_{k})\sigma_{k}^{-k}$ and $v_{k}'(0)=0$. For the final term we use the substitution $w=\theta\sqrt{-v_{k}''(0)}$ to see that \begin{equation}
										\int_{-k^{-\delta}}^{k^{-\delta}}\exp\left(v_{k}''(0)\frac{\theta^2}{2}\right)d\theta=\frac{1}{\sqrt{-v_k''(0)}}\int_{-\sqrt{-v_k''(0)}k^{-\delta}}^{\sqrt{-v_k''(0)}k^{-\delta}}\exp\left(-\frac{w^2}{2}\right)dw.
									\end{equation} The proof then follows from \eqref{eq:v'' exact}.
								\end{proof}

								Let $\Gamma_{2}=\Gamma\setminus\Gamma_1$. For $z\in\Gamma_{2}$ \begin{equation}\label{eq:real part of a circle}
									\frac{\Re(z)}{|\sigma_{k}|}>\cos(\pi\pm k^{-\delta})=-1+\frac{k^{-2\delta}}{2}+O\left( k^{-4\delta}\right).
								\end{equation} In the remainder of the proof we aim only to control the relative size of $g(\sigma_{k})$. Thus for convenience we use the principal branch, $\ln$, of the logarithm below.   Integrating $\Re\phi_{k}'$ along the arc from $\sigma_{k}$ to $\sigma_{k}e^{ik^{-\delta}}$ and using \eqref{eq:real part of a circle}  we see that \begin{equation}
									\begin{aligned}
										\ln \left|g\left(\sigma_{k}e^{ik^{-\delta}}\right)\right|-\ln g(\sigma_{k})&\lesssim -k^{1-2\delta}.
									\end{aligned}
								\end{equation} Thus, $\sup_{z\in \Gamma_{2}}\frac{|g(z)|}{g(\sigma_{k})}=O(\exp(-k^{\eps}) )$ for some sufficiently small $\eps>0$.

								It then follows that \begin{equation}\label{eq:elem asy}
									e_{k}=(-1)^{k}g(\sigma_{k})\sigma_{k}^{-k}(2\pi\rho k)^{-1/2}(1+o(1)).
								\end{equation}  Additionally, note that \begin{equation}
									\frac{d}{dz}\ln g(z)=s(z),
								\end{equation} and $\sigma_{k}$ is defined to be a solution to \begin{equation}
									s(\sigma_{k})=\frac{k}{\sigma_{k}}.
								\end{equation} One can see that\begin{equation}
									(k+1)\left[\frac{\sigma_{k}}{\sigma_{k+1}}-1 \right]\geq \ln g(\sigma_{k})-\ln g(\sigma_{k+1})
									\geq k\left[1-\frac{\sigma_{k+1}}{\sigma_{k}} \right].
								\end{equation} It then follows that \begin{equation}\label{eq:max of log g}
									\begin{aligned}
										\ln g(\sigma_{k})-\ln g(\sigma_{k+1})+k\ln|\sigma_{k}|-k\ln|\sigma_{k+1}|=o(1).
									\end{aligned}
								\end{equation} 
								

									We can then apply \eqref{eq:sigma asy}, \eqref{eq:elem asy}, and  \eqref{eq:max of log g} to see 
									\begin{equation}\label{eq:ratio of elem final}
										\begin{aligned}
											\frac{e_k}{e_{k+1}}&\sim_{k} -\frac{g(\sigma_{k})\sigma_{k}^{-k}(2\pi\rho k)^{-1/2}}{g(\sigma_{k+1})\sigma_{k+1}^{-k-1}(2\pi\rho (k+1))^{-1/2}}\\
											&\sim_{k}- \frac{g(\sigma_{k})}{g(\sigma_{k+1})} \left(\frac{\sigma_{k+1}}{\sigma_{k}} \right)^{k}\sigma_{k+1} \\
											&\sim_{k} \Cr \left[n_{+}^{-1}\left(\frac{k+1}{\pi\rho\csc(\pi\rho)}\right)\right]^{2}.\nc
										\end{aligned}
									\end{equation} The proof of \eqref{eq:coeff ratio} is then completed by noting $e_{k}=(-1)^{k}\gamma_{2k}/(2k)!$.
									As mentioned at the beginning of the section, \eqref{eq:ratio of coeff ratio condition} follows from \eqref{eq:coeff ratio}, so the proof of Lemma \ref{lem:coefficients from roots} is completed.								
								\end{proof}
								\section{Examples} \label{sec:examples} In this section, we give some examples of polynomials and entire functions which satisfy the assumptions of our main results. Of course the most important polynomial examples for our purposes are the even Jensen polynomials of functions satisfying Assumption \ref{assump:root density}. As we have already discussed in Section \ref{sec:finite free limit theorems}, random polynomials with iid roots and characteristic polynomials of random matrices serve as examples existing in the literature. We now discuss a few other examples.
								
								We choose the 130 degree polynomial, $p_{130}$, with roots placed, somewhat arbitrarily at $\left( ( (-26,26) \cap \Z/2 ) \cup ( (-27,27) \cap ((\Z/2)^{3/2} +\{3/2\})  ) \right)\setminus [-2,0] $, where $(\Z/2)^{3/2}$ denotes all numbers of the form $\pm(|k|/2)^{3/2}$ for some $k \in \Z$ and $+$ denotes the Minkowski sum.
								In Figure \ref{plot-F}, we plot the $114^{th}$, $122^{nd}$, and $126^{th}$ derivative of $p_{130}$, after rescaling $p_{130}$ so the its empirical root measure has mean $0$ and variance $130$, and normalizing the derivative to make it a monic polynomial. We also plot the corresponding Hermite polynomials.

								\begin{figure}[t] 
									\centering
									\begin{subfigure}{.45\textwidth}
										\centering
										\includegraphics[width=\linewidth]{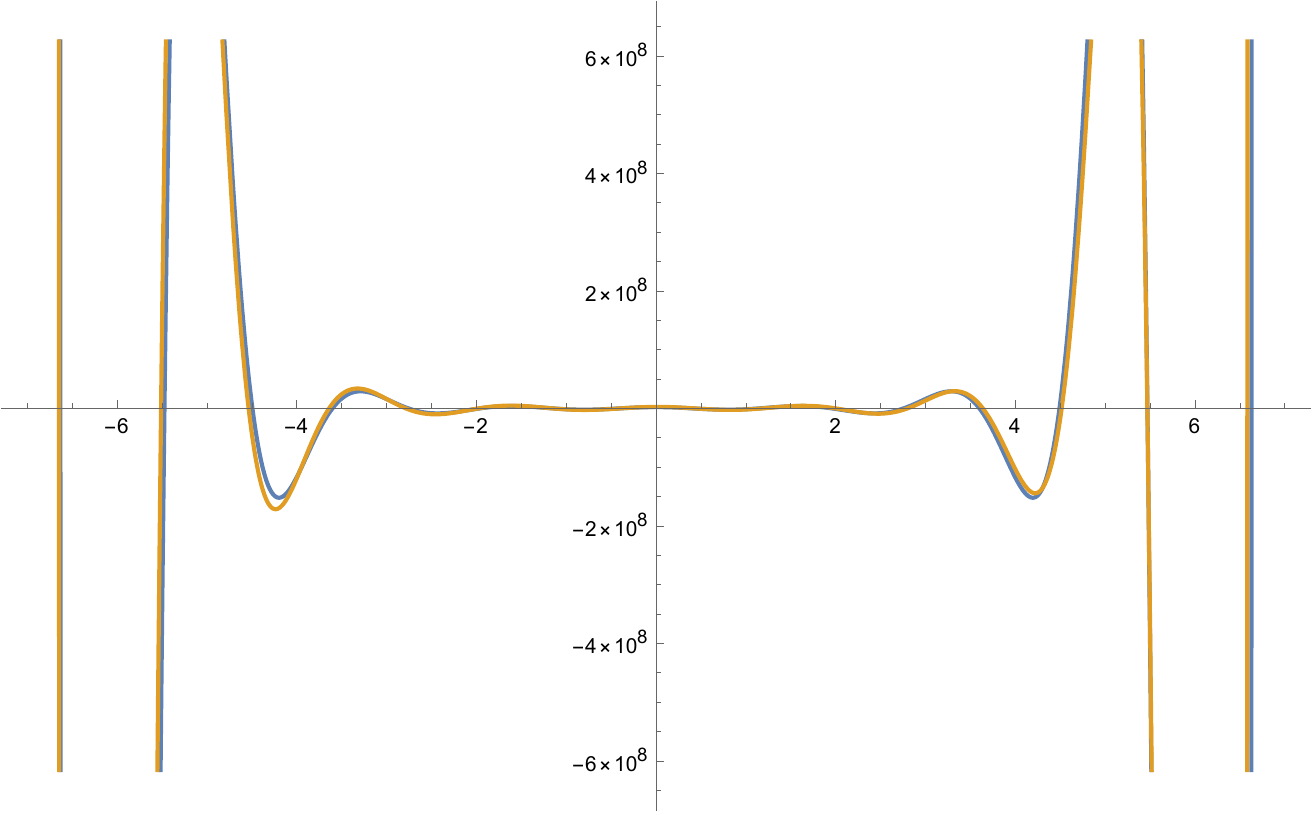}
										\caption{$114^{th}$ derivative}
										\label{poly16}
									\end{subfigure}%
									\hspace{0.5cm}
									\begin{subfigure}{.45\textwidth}
										\centering
										\includegraphics[width=\linewidth]{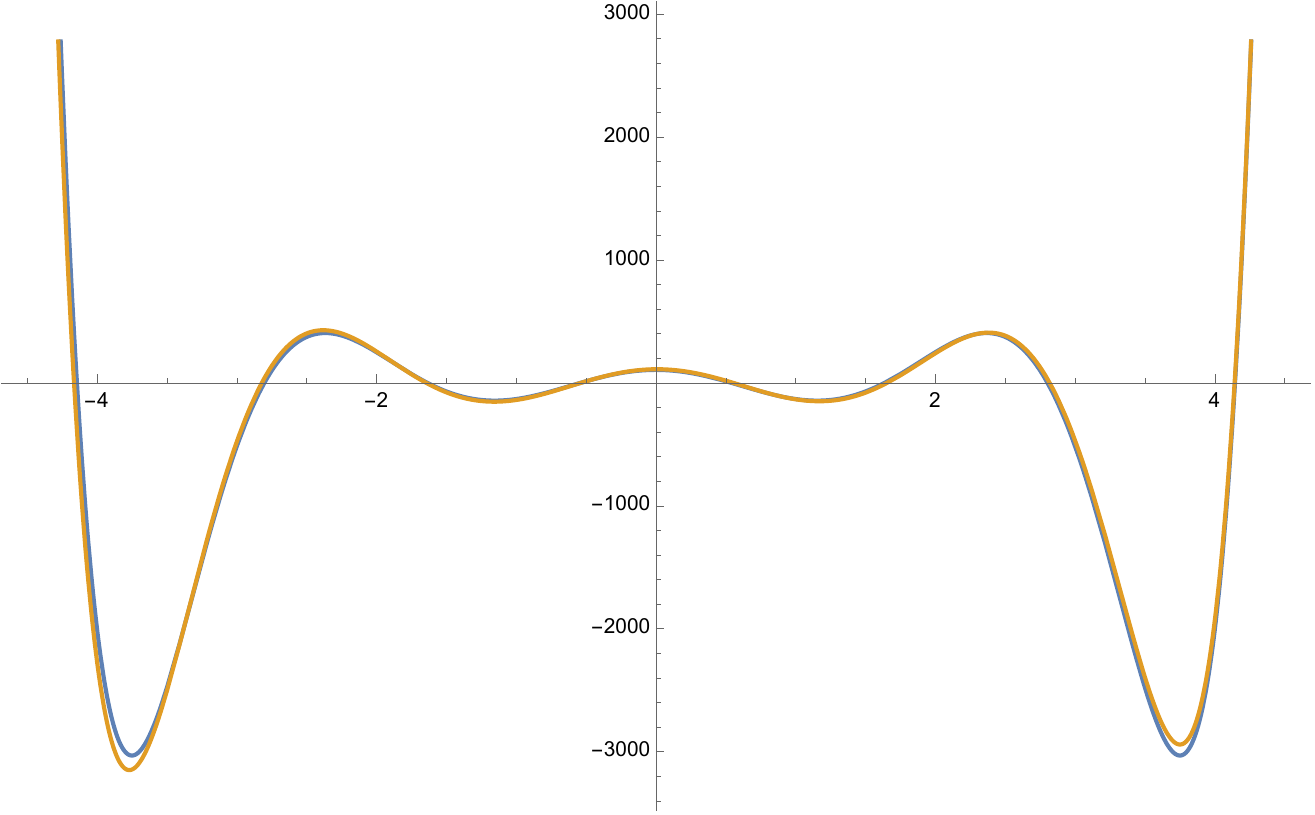}
										\caption{$122^{nd}$ derivative}
										\label{poly8}
									\end{subfigure}\\
									\begin{subfigure}{.45\textwidth}
										\centering
										\includegraphics[width=\linewidth]{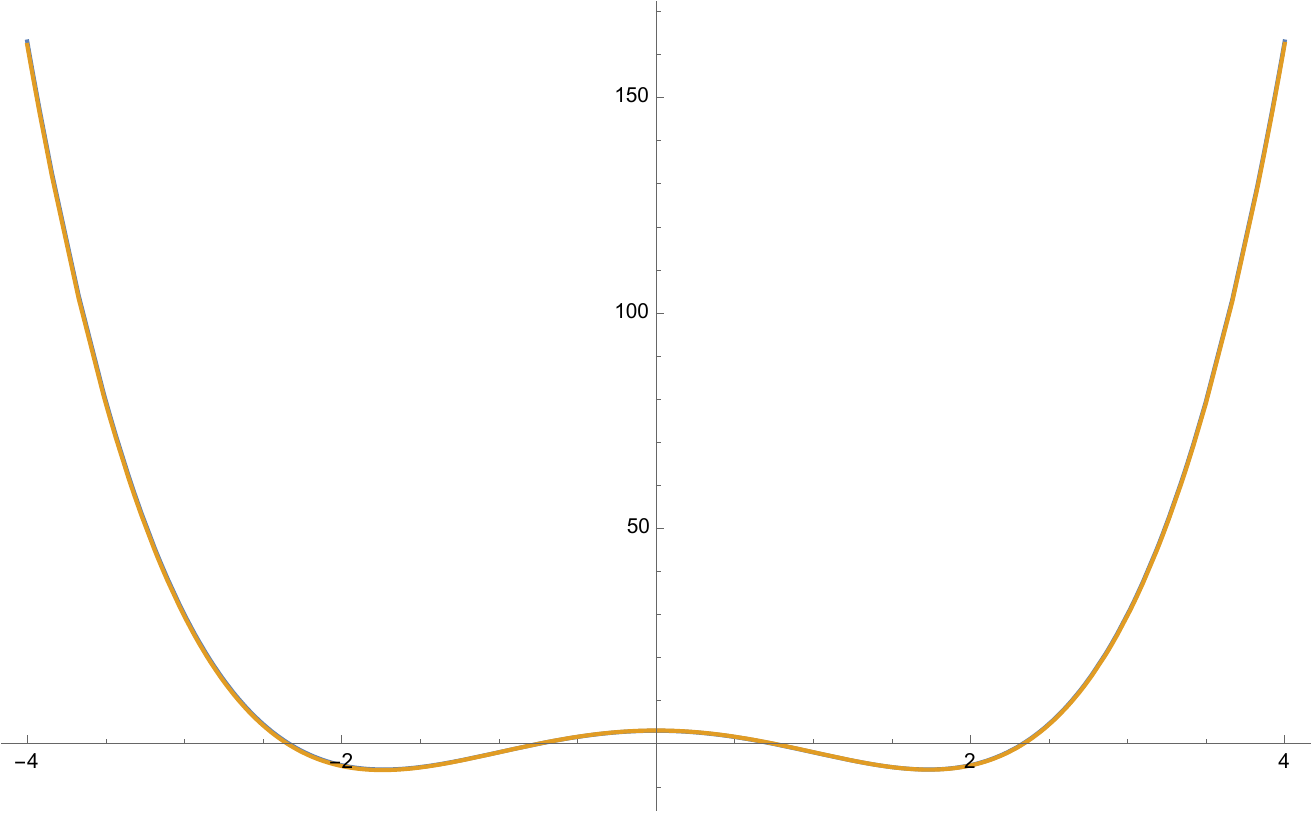}
										\caption{$126^{nd}$ derivative}
										\label{poly4}
									\end{subfigure}
									\caption{Derivatives of $p_{130}$ (blue) described in Section \ref{sec:examples} compared to the Hermite polynomials (orange).}
									\label{plot-F}
								\end{figure} 
								
								For entire functions, we first follow the lead of \cite{Pemantle-Subramanian2017}. Let $\beta\in(0,1)$ and $N_{\beta}=\{y_{j}\}_{j=1}^\infty$ be a Poisson point process with intensity measure having density $\beta x^{-(1+\beta)}$ on $(0,\infty)$. We define the entire function $g_{\beta}$ as \begin{equation}
									g_{\beta}(z)=\prod_{k=1}^{\infty} \left(1-y_{k}z\right).
								\end{equation} Then, one can check using standard properties of the Poisson point processes $N_{\beta}$ (see for example \cite{Davydov-Egorov2006,LePage-Woodroofe-Zinn1981}) that the random entire function $f_{\alpha}(z)=g_{\alpha/2}\left(z^{2}\right)$ almost surely satisfies Assumption \ref{assump:root density} for any $\alpha\in (0,2)$. Hence, Cosine Universality holds almost surely for $f_{\alpha}$, as does Hermite and Laguerre Universality for its Jensen polynomials  via Theorems \ref{thm:Cosine uni}, \ref{thm:Laguerre uni}, and \ref{thm:Hermite uni}. The case $\alpha=1$ is exactly an even version of the random entire function considered in \cite{Pemantle-Subramanian2017}. 
								
								For a deterministic example, we consider the Bessel function of the first  kind  $J_{\nu}$. There are multiple definitions for $J_{\nu}$, and we consider the series definition \begin{equation}
									J_{\nu}(z)=\left( \frac{1}{2}z\right)^{\nu}\sum_{k=0}^{\infty}\frac{(-1)^k}{4^{k}k!\Gamma(\nu+k+1)}z^{2k}
								\end{equation}  for some $\nu\geq 0$. For this choice of $\nu$ the zeros of $J_{\nu}$ are all real. We remove the (potential) branch point at $z=0$ and consider the functions \begin{equation}
									\widetilde{J}_{\nu}(z)=\sum_{k=0}^{\infty}\frac{(-1)^k}{4^{k}k!\Gamma(\nu+k+1)}z^{2k}.
								\end{equation} It is known \cite[Chapter 10.21]{NIST:DLMF} that if $j_{k,\nu}$ is the $k^{th}$ positive root of $\widetilde{J}_{\nu}$, then \begin{equation}\label{eq:Bessel root growth}
									j_{k,\nu}\sim \pi\left(k+\frac{\nu}{2}-\frac{1}{4} \right)
								\end{equation} as $k \to \infty$.  Assumption \ref{assump:root density} follows from \eqref{eq:Bessel root growth}. Thus, Cosine, Hermite, and Laguerre Universality hold for $\widetilde{J}_{\nu}$ and its even Jensen polynomials  via Theorems \ref{thm:Cosine uni}, \ref{thm:Laguerre uni}, and \ref{thm:Hermite uni}. 
								
								We conclude our examples by mentioning the work of Assiotis \cite{Assiotis2022} where random functions in the Laguerre--P\'olya class  are expressed as the limit of characteristic polynomials of unitarily invariant random Hermitian matrices. First, many of the examples discussed in \cite{Assiotis2022} could be taken as our choice of function $g(z)$ in the relation $f(z)=g\left(z^2\right)$, and our results would hold almost surely. Second, our work presents a different connection between random matrices and the Laguerre--P\'olya class. While \cite{Assiotis2022} considers scaling limits to random functions, we instead use finite free probability (where one averages over unitarily invariant random matrix ensembles) to provide a deterministic application to the Laguerre--P\'olya class. While we do not explore possible deeper connections between \cite{Assiotis2022} and our work here, we point out the work of Gorin and Marcus \cite{Gorin-Marcus2020} connecting $\beta$-ensembles in random matrix theory (which are much of the motivation for \cite{Assiotis2022}) and finite free probability.

	\bibliography{CosineUniv}
	\bibliographystyle{abbrv}

\end{document}